\documentclass[11pt,reqno,a4paper]{amsart}
\usepackage[toc,page]{appendix}

\usepackage[utf8]{inputenc}
\usepackage{esint}
\usepackage{MnSymbol}
\usepackage{enumitem}
\usepackage[english]{babel}
\usepackage{amsmath}
\usepackage{mathtools}
\usepackage{thm-restate}
\usepackage{comment}
\usepackage{tikz}
\usetikzlibrary{patterns}

\usepackage{xcolor}

\newtheorem{theorem}{Theorem}
\newtheorem*{theorem*}{Theorem}
\newtheorem{proposition}{Proposition}[section]
\newtheorem{lemma}[proposition]{Lemma}

\theoremstyle{definition}
\newtheorem{definition}[proposition]{Definition}
\newtheorem{remark}[proposition]{Remark}

\numberwithin{equation}{section}

\usepackage[
colorlinks,
pdfpagelabels,
pdfstartview = FitH,
bookmarksopen = true,
bookmarksnumbered = true,
linkcolor = blue,
plainpages = false,
hypertexnames = false,
citecolor = red] {hyperref}

\hypersetup{
linktoc=page
}

\usepackage[capitalize]{cleveref}
\crefname{enumi}{}{}
\crefname{equation}{}{}

\makeatletter
\def\@tocline#1#2#3#4#5#6#7{\relax
\ifnum #1>\c@tocdepth % then omit
\else
\par \addpenalty\@secpenalty\addvspace{#2}%
\begingroup \hyphenpenalty\@M
\@ifempty{#4}{%
\@tempdima\csname r@tocindent\number#1\endcsname\relax
}{%
\@tempdima#4\relax
}%
\parindent\z@ \leftskip#3\relax \advance\leftskip\@tempdima\relax
\rightskip\@pnumwidth plus4em \parfillskip-\@pnumwidth
#5\leavevmode\hskip-\@tempdima
\ifcase #1
\or\or \hskip 1em \or \hskip 2em \else \hskip 3em \fi%
#6\nobreak\relax
\dotfill\hbox to\@pnumwidth{\@tocpagenum{#7}}\par
\nobreak
\endgroup
\fi}
\makeatother

\def \R {\mathbb {R}}

\def \E {\mathbb{E}}
\def \P {\mathbb{P}}

\def\supp{\operatorname{supp}}
\def\osc{\operatorname{osc}}

\renewcommand{\tilde}{\widetilde}

\newcommand{\I}{\mathbb{I}}
\newcommand{\Ha}{\mathcal{H}}
\newcommand{\eps}{\varepsilon}

\newcommand{\pa}{\partial}

\renewcommand{\fint}{\strokedint}

\usepackage{etoolbox}

\newcommand{\capacity}[2]{
\ifstrequal{#1}{(}{\operatorname{cap}#1#2}{\operatorname{cap}(#1,#2)}
}

\allowdisplaybreaks

\begin{document}

\title[Geometric Characterization of the Eyring-Kramers Formula]{Geometric Characterization of the Eyring-Kramers Formula}
\author[Avelin]{Benny Avelin}
\address{Benny Avelin,
Department of Mathematics,
Uppsala University,
S-751 06 Uppsala,
Sweden}
\email{\color{blue} benny.avelin@math.uu.se}
\author[Julin]{Vesa Julin}
\address{Vesa Julin,
Department of Mathematics and Statistics,
University of Jyv\"askyl\"a,
P.O. Box 35,
40014 Jyv\"askyl\"a,
Finland}
\email{\color{blue} vesa.julin@jyu.fi}
\author[Viitasaari]{Lauri Viitasaari}
\address{Lauri Viitasaari,
Department of Mathematics,
Uppsala University,
S-751 06 Uppsala,
Sweden}
\email{\color{blue} lauri.viitasaari@math.uu.se}


\date{\today}

\begin{abstract}
  In this paper we consider the mean transition time of an over-damped Brownian particle between local minima of a smooth potential. When the minima and saddles are non-degenerate this is in the low noise regime exactly characterized by the so called Eyring-Kramers law and gives the mean transition time as a quantity depending on the curvature of the minima and the saddle. In this paper we find an extension of the Eyring-Kramers law giving an upper bound on the mean transition time when both the minima/saddles are degenerate (flat) while at the same time covering multiple saddles at the same height. Our main contribution is a new sharp characterization of the capacity of two local minimas as a ratio of two geometric quantities, i.e.,  the smallest separating surface and  the geodesic distance.
\end{abstract}

\maketitle
\setcounter{tocdepth}{2}
\tableofcontents

\section{Introduction} %\label{sec:intro}
In this paper we investigate the so called metastable exit times for the stochastic differential equation
\begin{align} \label{eq:SGD}
  dX_t = -\nabla F(X_t) dt + \frac{\sqrt{\varepsilon}}{2} dB_t,
\end{align}
where $F$ is a smooth potential with many local minimas and $\varepsilon$ is a small number.

The main question of metastability is to determine how long time does the process \cref{eq:SGD} take from going from one local minima to another one. We call these the metastable exit times. This question has a rich history and in the double well case with non-degenerate minimas and a saddle point this is characterized by a formula called Eyring-Kramers law \cite{Eyr,Kra} which can be stated as follows: Assume that $x$ and $y$ are quadratic local minimas of $F$, separated by a unique saddle $z$ which is such that the Hessian has a single negative eigenvalue $\lambda_1(z)$. Then the expected transition time from $x$ to $y$ satisfies
\begin{align} \label{e.kramer}
  \E^x[\tau] \simeq \frac{2\pi}{|\lambda_1(z)|} \sqrt{\frac{\det(\nabla^2 F(z))}{\det(\nabla^2 F(x))}} e^{(F(z)-F(x))/\eps},
\end{align}
where $\simeq$ denotes that the comparison constant tends to $1$ as $\eps \to 0$.

The validity of the above formula has been studied, from a qualitative perspective, quite extensively, starting from the work of Freidlin and Wentzell. For more information, see the book \cite{FW}. Roughly 15 years ago, Bovier et. al. produced a series of papers \cite{BGK01,BGK02,BGK,BGK2} (see also \cite{BGKBook}) which provided the first proof of \cref{e.kramer} in the general setting of Morse functions. Specifically, they showed that the comparison function is like $1+\mathcal{O}(\eps^{1/2}|\log\eps|^{3/2})$. In these papers, they utilized the connection to classical potential theory in order to reduce the problem of estimating metastable exit times to the problem of estimating certain capacities sharply. This approach was later used in \cite{Berglund} to generalize \cref{e.kramer} to general polynomial type of degeneracies.

In this paper we are interested in  estimating the metastable exit times in the case of general type of degenerate critical points. This requires new techniques and effective notation from geometric function theory which we will describe below. Our motivation comes from the field of non-convex optimization where we cannot expect the minimas/saddles to be quadratic or even to have polynomial growth in any direction. In particular, such situations are well known in the context of neural networks, where the minimas and saddles may be completely flat in some directions \cite{HS}. Furthermore, it seems that they are preferrable, see \cite{PKA} for a discussion, see also \cite{AJ} for an explicit example.

The main goal is to estimate the dependency of the metastable exit times with respect to the geometry of the potential $F$. In the proof of \cref{e.kramer} in \cite{BGK} this is reduced to estimating the ratio of the $L^1$ norm of the hitting probability and the capacity. Thus in order to estimate the metastable exit times, one needs to produce
\begin{enumerate}
  \item \label{en:estimate1} Estimates of the integral of the hitting probability, i.e. the integral of capacitary potentials with respect to the Gibbs measure.
  \item \label{en:estimate2} Estimates of the capacity itself, i.e. estimates of the energy of the capacitary potentials.
\end{enumerate}
The interesting point is that the influence of $F$ on \cref{en:estimate1,en:estimate2} is in a sense dual. Specifically, the shape of minimas of $F$ influence \cref{en:estimate1} while the shape of saddles between minimas influence \cref{en:estimate2}. As is well known, the main difficulty is to estimate \cref{en:estimate2}, which is  an interesting topic of its own.

Our main contribution is a sharp capacity estimate for a very general class of degenerate saddle points. In order to achieve this, we phrase the problem in the language of geometric function theory, where the capacity estimates are  central topic \cite{G, Z}. We introduce two geometric quantities which allow us to estimate the capacity in a sharp and natural way. As a byproduct, we see that in the case of several saddle points at the same height, the topology dictates how the local capacities add up. Specifically, we consider two cases, which we call the parallel and the serial case, and it turns out that the formulas for the total capacity have natural counterparts in electrical networks of capacitors, see \cref{thm1}. Even in the context of non-degenerate saddles, our formulas provide a generalization of the result of \cite{BGK} where the authors consider only the parallel case.  As we mentioned, we allow the saddle points to be degenerate but we have to assume that saddles are non-branching, see \cref{eq:struc3}.

\subsection{Assumptions and statement of the main results} \label{sec:ass}

In order to state our main results we first need to introduce our assumptions on the potential $F$. We also need to introduce  notation from geometric function theory
which might seem rather heavy at first, but it turns out to be robust enough for us to treat the potentials with possible degenerate critical points.

Let us first introduce some general terminology. We say that a critical point $z$ of a function $f \in C^1(\R^n)$ is a local minimum (maximum) of $f$ if $f(x) \geq f(z)$  ($f(x) \leq f(z)$) in a neighborhood of $z$. If $f$ is not locally constant at a critical point $z$, then $z$ is a saddle point if it is not a local minimum / maximum.
For technical reasons we also allow saddle points to include points $z$ where $f$ is locally constant. We say that  a local minimum at  $z$  is \emph{proper} if there exists a $\hat \delta > 0$ such that for every  $0 < \delta < \hat \delta$ there exists a $\rho$ such that
\begin{align*}
  f(x) \geq f(z) + \delta \quad \text{for all } \, x \in \pa B_{\rho}(z),
\end{align*}
where $B_\rho(z)$ denotes open ball with radius $\rho$ centered at $z$. When the center is at the origin we use the short notation $B_\rho$.

Let us then proceed to our assumptions on the potential $F$.  Throughout the paper we assume  that $F \in C^2(\R^n)$ and satisfies the following quadratic growth  condition
\begin{equation} \label{eq:struc0}
  F(x) \geq \frac{ |x|^2}{C_0}  - C_0
\end{equation}
for a  constant $C_0 \geq 1$. We assume that every local minimum point $z$ of $F$ is proper, as described above, and that there is a convex function $G_z: \R^n \to \R$ which has a proper minimum at $0$ with $G(0)= 0$ such that
\begin{equation} \label{eq:struc1}
  \big| F(x+z)- F(z)- G_z(x) \big|\leq \omega \big( G_z(x)\big),
\end{equation}
where $\omega : [0,\infty) \to [0,\infty)$ is a continuous and increasing function  with
\begin{equation} \label{eq:struc11}
  \lim_{s \to 0} \frac{\omega(s)}{s} = 0.
\end{equation}
We denote by  $\delta_0$ the largest number  for which $\omega(\delta) \leq \frac{\delta}{8}$ for all $\delta \leq 4 \delta_0$. We define a neighborhood of the local minimum point $z$ and $\delta < \delta_0$ as
\begin{equation*}% \label{eq:struc2}
  O_{z,\delta} :=   \{ x \in \R^n:   G_z(x) <  \delta\} +\{ z\}.
\end{equation*}

For the saddles, we assume that for every saddle point $z$  of $F$ there are convex functions $g_z: \R \to \R$ and $G_z :\R^{n-1} \to \R$ which have proper minimum at $0$ with $g_z(0) = G_z(0) = 0$, and that there exists an isometry $T_z : \R^n \to \R^n$  such that, denoting $x = (x_1, x') \in \R \times \R^{n-1}$, it holds
\begin{equation} \label{eq:struc3}
  \big| (F\circ T_z) (x) -F(z)  + g_z(x_1) - G_z(x'))\big|\leq \omega( g_z(x_1))  + \omega( G_z(x')),
\end{equation}
where $\omega : [0,\infty) \to [0,\infty)$ is as in \cref{eq:struc11}. The assumption \cref{eq:struc3} allows the saddle point to be degenerate, but we do not allow them to have many branches, i.e., the sets $\{F < F(z)\}\cap B_{\rho}(z)$ cannot have more than two components.  Note that the convex functions $g_z, G_z$ and the isometry $T_z$ depend on $z$, while the function $\omega$ is the same for all saddle points.   We define a neighborhood of the saddle point $z$ and $\delta < \delta_0$ as
\begin{equation} \label{eq:struc4}
  O_{z,\delta} := T_z^{-1}\left(  \{x_1 \in \R:  g_z(x_1) < \delta\} \times \{x' \in \R^{n-1}:  G_z(x') < \delta\}\right),
\end{equation}
where $T_z$ is the isometry in \cref{eq:struc1}. Note that, since the saddle may be flat, we should talk about sets rather than points. However, we adopt the convention that we always  choose a representative point from each saddle (set) and thus we may label the saddles by points $z_1, z_2, \dots$. Moreover, we assume that there is a $\delta_1 \leq \delta_0$ such that for $\delta < \delta_1$ we have that if $z_1$ and $z_2$ are two different saddle points, then their neighborhoods $O_{z_1, 3\delta}$ and   $O_{z_2, 3\delta}$ defined in \cref{eq:struc4} are disjoint.  We assume the same for local minimas (or more precisely, the representative points of sets of local minimas).

Let us then introduce the notation related to  geometric function theory, see \cite{G,Z}.  Let us  fix two disjoint sets $A$ and $B$ in a domain $\Omega$ (open and connected set). We say that  a smooth path $\gamma :[0,1] \to \R^n$ \emph{connects $A$ and $B$ in the domain $\Omega$}  if
\begin{align*}
  \gamma(0) \in A, \quad \gamma(1) \in B \quad \text{and} \quad \gamma([0,1]) \subset  \Omega.
\end{align*}
In this case we denote $\gamma \in \mathcal{C}(A,B; \Omega)$. We follow the standard notation from geometric function theory and define a dual object to this by saying that a smooth hypersurface $S \subset \R^n$ (possibly with boundary)  \emph{separates $A$ from $B$ in  $\Omega$} if every path
$\gamma \in \mathcal{C}(A,B; \Omega)$  intersects $S$. In this case we denote  $S \in \mathcal{S}(A,B;\Omega)$.  We define the geodesic distance between $A$ and $B$ in $\Omega$ as
\begin{equation} \label{eq:struc5}
  d_{\eps}(A,B; \Omega) := \inf\left( \int_{\gamma} |\gamma'| e^{\frac{F(\gamma)}{\eps}} \, dt  : \gamma \in \mathcal{C}(A,B; \Omega)  \right)
\end{equation}
and   its dual  by
\begin{equation} \label{eq:struc6}
  V_{\eps}(A,B; \Omega)   := \inf\left( \int_{S}  e^{-\frac{F(x)}{\eps}} \, d \mathcal{H}^{n-1}(x) :S \in  \mathcal{S}(A,B;\Omega) \right)  .
\end{equation}
Here $\mathcal{H}^{k}$ denotes the $k$-dimensional Hausdorff measure. Finally, we define  the communication height between the sets $A$ and $B$ as
\begin{equation*} %\label{eq:struc77}
  F(A; B) := \inf_{\gamma \in \mathcal{C}(A,B; \R^n)} \sup_{t \in [0,1]}  \,  F(\gamma(t)).
\end{equation*}

Let us then assume that $x_a$ and $x_b$ are  local minimum points   and  denote the height of the saddle between $x_a$ and $x_b$ as
\begin{align*}
  F(x_a; x_b) :=  F(\{x_a\}; \{x_b\}).
\end{align*}
Notice that  $F(x_a), F(x_b) \leq F(x_a; x_b)$. For $s \in \R$, denote
\begin{align*}
  U_s : = \{ x \in \R^n : F(x) <  F(x_a; x_b)  + s\}.
\end{align*}
We note that the points $x_a$ and $x_b$ lie in different components of the set $U_{-\delta/3}$ while they are in the same component of the set  $U_{\delta/3}$. We  will always denote the   components of $U_{-\delta/3}$ containing the points $x_a$ and $x_b$ by $U_a$ and $U_b$, respectively.   It is important to notice that
if $z$ is a saddle point and  $F(z) < F(x_a; x_b) + \delta/3$, then the  neighborhood  $O_{z,\delta}$ defined in  \cref{eq:struc4} intersects the set  $U_{-\delta/3}$.
We will sometimes call the components of the set $U_{-\delta/3}$ \emph{islands} and the neighborhoods  $O_{z,\delta}$ \emph{bridges} since we may connect islands with bridges, see \cref{f:localization}. (The terminology is obviously taken from Seven Bridges of K\"onigsberg).
We say that the set of saddle  points $Z_{x_a,x_b}= \{z_1, \dots, z_N\}$  \emph{charge capacity} if it is the smallest set with the property that  every  $\gamma \in \mathcal{C}(B_\eps(x_a),B_\eps(x_a); U_{\delta/3})$  intersects the bridge  $O_{z_i,\delta}$, defined in  \cref{eq:struc4}, for some $z_i \in Z_{x_a,x_b}$. In particular, it holds that $Z_{x_a,x_b} \subset U_{\delta/3}$.

\begin{figure}
  \begin{center}
    \begin{tikzpicture}

      \draw (-2,-2) .. controls (0,-1) and (0,1) .. (-2,2);
      \draw (-2,1) node {$U_b$};
      \draw (2,-2) .. controls (0,-1) and (0,1) .. (2,2);
      \draw (2,1) node {$U_a$};

      %\draw[dashed] (0,-2) -- (0,2);
      \filldraw (0,0) circle (0.1);
      \draw (0.25,-0.2) node {$z$};

      \filldraw[color=gray!60, fill=gray!30] (-2,0) circle (0.5);
      \draw (-2,0) node {$x_b$};
      \draw (-2,-0.8) node {$B_\varepsilon(x_b)$};
      \filldraw[color=gray!60, fill=gray!30] (2,0) circle (0.5);
      \draw (2,0) node {$x_a$};
      \draw (2,-0.8) node {$B_\varepsilon(x_a)$};

      \draw (-0.8,-1.3) rectangle (0.8,1.3);
      \draw (0,-1.7) node {$O_{z, \delta}$};
    \end{tikzpicture}
  \end{center}
  \caption{The neighborhood $O_{z, \delta}$ of the saddle point $z$ connects the sets $U_a$ and $U_b$. }
  \label{f:localization}
\end{figure}
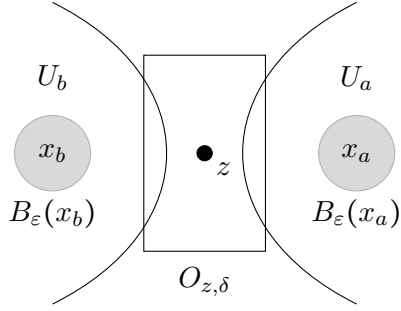

We will focus on two different topological situations, where the saddle points in $Z_{x_a,x_b} $ are either \emph{parallel} or in \emph{series}.  We say that the
points in $Z_{x_a,x_b}$ are parallel if for every $z_i \in Z_{x_a,x_b}$ there is a path
\begin{align*}
  \gamma \in \mathcal{C}(B_\eps(x_a),B_\eps(x_b); U_{\delta/3})
\end{align*}
passing only through $z_i$.   We say that the  points in $Z_{x_a,x_b}$ are in series if every path  $\gamma \in \mathcal{C}(B_\eps(x_a),B_\eps(x_b); U_{\delta/3})$ passes through   the bridge  $O_{z_i,\delta}$, defined in \cref{eq:struc4},  for all $z_i \in Z_{x_a,x_b}$. In other words, if the points in $Z_{x_a,x_b}= \{z_1, \dots, z_N\}$ are parallel, then the islands occupied by  the points $x_a$ and $x_b$ respectively  are  connected with $N$ bridges and we need to pass only one to get from $x_a$ to $x_b$. If they are in series, then we have to pass all $N$ bridges  in order to get from $x_a$ to $x_b$, see \cref{f:par:ser}.

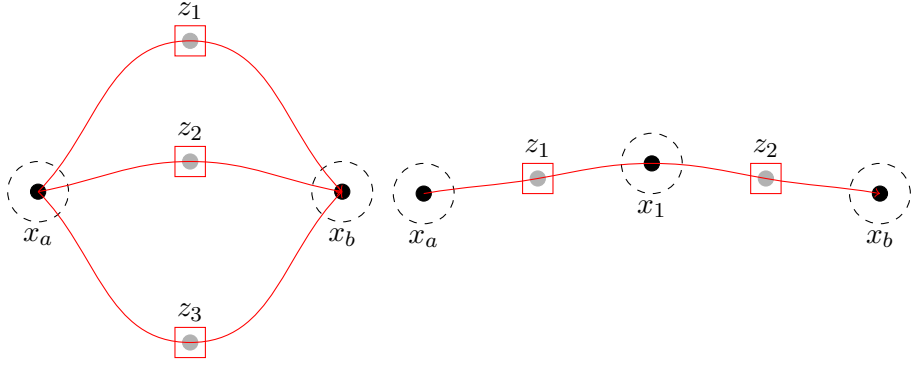
\begin{figure}
  \begin{center}
    \begin{minipage}{6in}
      \raisebox{-0.45\height}{
      \begin{tikzpicture}[scale=2]
        % Lets draw the two minima

        \filldraw (-1,0) circle (0.05);
        \draw[dashed] (-1,0) circle (0.2);
        \draw (-1,-0.3) node {$x_a$};
        \filldraw (1,0) circle (0.05);
        \draw[dashed] (1,0) circle (0.2);
        \draw (1,-0.3) node {$x_b$};

        % Lets draw the three saddles
        \filldraw[color=gray!60] (0,1) circle (0.05);
        \draw (0,1.2) node {$z_1$};
        \filldraw[color=gray!60] (0,0.2) circle (0.05);
        \draw (0,0.4) node {$z_2$};
        \filldraw[color=gray!60] (0,-1) circle (0.05);
        \draw (0,-1+0.2) node {$z_3$};

        % curve going right
        \draw[->,color=red] (-1,0) to[out=45,in=-180] (0,1) to[out=0,in=-180-45] (1,0);
        \draw[->,color=red] (-1,0) to[out=10,in=-180] (0,0.2) to[out=0,in=-180-10] (1,0);
        \draw[->,color=red] (-1,0) to[out=-45,in=-180] (0,-1) to[out=0,in=-180+45] (1,0);

        \draw[red,rotate around={0:(0,0)}] (-0.1,1+0.1) rectangle (0.1,1-0.1);
        \draw[red,rotate around={0:(0,0)}] (-0.1,0.2+0.1) rectangle (0.1,0.2-0.1);
        \draw[red,rotate around={0:(0,0)}] (-0.1,-1+0.1) rectangle (0.1,-1-0.1);

      \end{tikzpicture}}
      \raisebox{-0.5\height}{
      \begin{tikzpicture}[scale=2]
        % Lets draw the two minima

        \filldraw (-1,0) circle (0.05);
        \draw[dashed] (-1,0) circle (0.2);
        \draw (-1,-0.3) node {$x_a$};
        \filldraw (0.5,0.2) circle (0.05);
        \draw[dashed] (0.5,0.2) circle (0.2);
        \draw (0.5,0.2-0.3) node {$x_1$};
        \filldraw (2,0) circle (0.05);
        \draw[dashed] (2,0) circle (0.2);
        \draw (2,-0.3) node {$x_b$};

        % Lets draw the three saddles
        \filldraw[color=gray!60] (-0.25,0.1) circle (0.05);
        \draw (-0.25,0.1+0.2) node {$z_1$};
        \filldraw[color=gray!60] (0.5+0.75,0.1) circle (0.05);
        \draw (0.5+0.75,0.1+0.2) node {$z_2$};
        % % curve going right
        \draw[->,color=red] (-1,0) to[out=10,in=-180+10] (-0.25,0.1) to[out=10,in=-180] (0.5,0.2) to[out=0,in=-180-10] (0.5+0.75,0.1) to[out=-10,in=-180-10] (2,0);

        \draw[red,rotate around={0:(0,0)}] (-0.25-0.1,0.1-0.1) rectangle (-0.25+0.1,0.1+0.1);
        \draw[red,rotate around={0:(0,0)}] (0.5+0.75-0.1,0.1-0.1) rectangle (0.5+0.75+0.1,0.1+0.1);
      \end{tikzpicture}}
    \end{minipage}
  \end{center}
  \caption{Left picture is the parallel case and the right is the series case.}
  \label{f:par:ser}
\end{figure}

Recall that  $U_a$ and $U_b$ denote the islands, i.e.,  the components of $U_{-\delta/3}$, which contain the points $x_a$ and $x_b$.
If the points in $Z_{x_a,x_b} = \{ z_1, \dots, z_N\}$ are parallel, then we may connect $U_a$ and $U_b$ with one bridge, i.e., for every $z_i$ the set
\begin{equation} \label{eq:struc7}
  U_{z_i,\delta} :=   O_{z_i, \delta} \cup U_a \cup U_b
\end{equation}
is connected, again see \cref{f:localization}. Then  all paths $\gamma \in \mathcal{C}(B_\eps(x_a),B_\eps(x_b); U_{z_i,\delta})$ pass through the bridge  $O_{z_i,\delta}$. If the points in  $Z_{x_a,x_b}$ are in series, then it is useful to order them $Z_{x_a,x_b} = \{ z_1, \dots, z_N\}$ as follows. Let us consider a path $\gamma  \in \mathcal{C}(B_\eps(x_a),B_\eps(x_b); U_{\delta/3})$ which  passes through each point in $Z_{x_a,x_b}$  precisely once. This means that there are
\begin{equation} \label{eq:struc8}
  0 < t_1 < \dots < t_N <1 \quad \text{ such that } \quad \gamma(t_i) = z_i,
\end{equation}
which gives a natural ordering for  points in  $Z_{x_a,x_b}$. By the assumption \cref{eq:struc3}, we also deduce that there are $s_1, \dots,  s_{N-1}$ such that $t_i < s_i < t_{i+1}$ and  $F(\gamma(s_i)) < F(x_a; x_b)  - \delta/3$. We denote
\begin{equation} \label{eq:struc9}
  \gamma(s_i) = x_i, \,\, x_0 = x_a \quad \text{ and  } \quad  x_{N} = x_b.
\end{equation}
The idea  is that then every point  $x_i$ lie in a different island, i.e., component of $U_{-\delta/3}$, see \cref{f:par:ser}.

We are now ready to state our main results. The first result is a quantitative lower bound on the capacity between the sets $B_\eps(x_a)$ and $B_\eps(x_b)$, where $x_a$ and $x_b$ are two local minimum points of $F$. For a given domain $\Omega \subset \R^n$ we define the capacity of two disjoint  sets $A, B \subset \Omega$ with respect to the domain $\Omega$ as

\begin{equation*}
  %\label{eq:capacity-main-results}
  \capacity(A,B; \Omega) := \inf \left( \eps \int_{\Omega } |\nabla u|^2 e^{-\frac{F}{\eps}}\, dx \, : \,\, u=1 \,\, \text{in } \, A,  \,\, u=0 \,\, \text{in } \, B \right).
\end{equation*}
Above, the infimum is taken over functions $u \in W_{loc}^{1,2}(\Omega)$.  In the case $\Omega = \R^n$ we denote
\begin{align*}
  \capacity(A,B) = \capacity(A,B; \R^n)
\end{align*}
for short.

Finally, for functions $f$ and $g$ which depend  continuously on $\eps>0$, we adopt the notation
\begin{align*}
  f(\eps) \simeq g(\eps)
\end{align*}
when there exists a constant $C$ depending only on the data of the problem such that
\begin{align*}
  (1-\hat \eta(C \eps)) f(\eps) \leq g(\eps) \leq (1+\hat \eta(C \eps)) f(\eps),
\end{align*}
where $\hat \eta$ is an increasing and continuous function $\hat \eta:[0,\infty) \to [0,\infty)$ with $\lim_{s \to 0} \hat \eta(s) = 0$. In all our estimates, the function $\hat \eta$ is specified and depends only on the function $\omega$ from \cref{eq:struc1,eq:struc3}.
In order to define it, we first let $0 < \eps \leq \delta_0/2$ be fixed and let $\eps_1(\eps)$ be the unique solution to
\begin{equation} \label{lem:geometric-1}
  \sqrt{\omega(\eps_1) \eps_1} = \eps.
\end{equation}
From the assumption that $\omega(s) < s/2$ for $s < \delta_0$ we see that $\eps < \eps_1$. Furthermore, since $\omega$ is increasing we get that $\eps_1 \to 0$ as $\eps \to 0$. Now, from the definition of $\eps_1$ in \cref{lem:geometric-1} we see, using $\lim_{s \to 0} \frac{\omega(s)}{s} = 0$ and $\eps_1 \to 0$ as $\eps \to 0$, that
\begin{align*}
  \frac{\eps_1}{ \eps} = \frac{\sqrt{\eps_1}}{\sqrt{\omega(\eps_1)}} \to \infty \qquad \text{as }\, \eps \to 0.
\end{align*}
On the other hand, again using the same facts, we see that
\begin{align*}
  \frac{\omega(\eps_1)}{ \eps} = \frac{\sqrt{\omega(\eps_1)}}{\sqrt{\eps_1}} \to 0\qquad \text{as }\, \eps \to 0.
\end{align*}
Thus
\begin{equation} \label{lem:geometric-2}
  \lim_{\eps \to 0} \frac{\eps_1}{ \eps} = \infty \quad \text{and} \quad \lim_{\eps \to 0} \frac{\omega(\eps_1)}{ \eps} = 0 .
\end{equation}
In the following we will denote
\begin{align}
  \label{e.eta} \eta(x) &= e^{-1/x} x^n \\
  \label{e.etahat} \hat \eta(C \eps) &= \eta\left (C \frac{\eps_1(\eps)}{\eps} \right ).
\end{align}

Finally, in our main theorems and our lemmas/propositions beyond \cref{sec:technical} there is a ball $B_R$ which contains all the level sets of interest. The existence of such a ball is given by the quadratic growth condition \cref{eq:struc0}. The constants in the estimates in our main theorems and in \cref{sec:technical} are unless otherwise stated, depending on $n,\|\nabla F\|_{B_R}, \delta, R, C_0$, specifically, this applies to the constants in $\hat \eta$, and as such, gives precise meaning to $a \simeq b$.

\begin{theorem}\label{thm1}
  Assume that $F$ satisfies the structural assumptions above. Let $x_a$ and $x_b$ be two local minimum points of $F$ and let  $Z_{x_a,x_b}= \{z_1, \dots, z_N\}$ be the set of saddle points which charge capacity as defined above, and let $0 < \delta \leq \delta_1$ be fixed. There exists an $0 < \eps_0 \leq \delta$ such that if $0 < \eps \leq \eps_0$ the following holds:

  If the points in  $Z_{x_a,x_b}= \{z_1, \dots, z_N\}$ are parallel, then, using the notation $U_{z_i,\delta}$  from  \cref{eq:struc7}, it holds
  \begin{equation} \label{eq:thm1-1}
    \capacity(B_\eps(x_a), B_\eps(x_b)) \simeq \sum_{i=1}^N \capacity(B_\eps(x_a), B_\eps(x_b); U_{z_i,\delta})  .
  \end{equation}
  Moreover  for all $i =1,\dots, N$ we have the estimate
  \begin{align*}
    \capacity(B_\eps(x_a), B_\eps(x_b); U_{z_i,\delta}) \simeq   \eps \frac{V_{\eps}(B_\eps(x_a),B_\eps(x_b); U_{z_i,\delta}) }{d_{\eps}(B_\eps(x_a),B_\eps(x_b); U_{z_i,\delta}) } e^{\frac{F(z_i)}{\eps}},
  \end{align*}
  where $d_{\eps}(B_\eps(x_a),B_\eps(x_b); U_{z_i,\delta}) $ and $V_{\eps}(B_\eps(x_a),B_\eps(x_b); U_{z_i,\delta})$ are defined in \cref{eq:struc5} and \cref{eq:struc6}.

  If  the points in  $Z_{x_a,x_b}$ are in series, then, using the ordering  $z_1, \dots, z_N$ from \cref{eq:struc8} for the points in $Z_{x_a,x_b}$ and the points $x_0, x_1, \dots, x_{N}$ defined in \cref{eq:struc9}, it holds
  \begin{equation} \label{eq:thm1-2}
    \frac{1}{\capacity(B_\eps(x_a), B_\eps(x_b))} \simeq  \sum_{i=1}^N \frac{1}{\capacity(B_\eps(x_{i-1}), B_\eps(x_{i}))},
  \end{equation}
  where we have the estimate
  \begin{align*}
    \capacity(B_\eps(x_{i-1}), B_\eps(x_i)) \simeq  \eps \frac{V_{\eps}(B_\eps(x_{i-1}), B_\eps(x_i))}{d_{\eps}(B_\eps(x_{i-1}),B_\eps(x_{i})) } e^{\frac{F(z_i)}{\eps}}
  \end{align*}
  for all $i =1,\dots, N$.
\end{theorem}

Let us make a few remarks on the statement of the above theorem. First, in the case of  a single saddle $Z_{x_a,x_b}= \{z\}$ the above capacity estimate reduces to
\begin{align*}
  \capacity{B_\eps(x_{a})}{B_\eps(x_b)} \simeq   \eps \frac{V_{\eps}(B_\eps(x_{a}),B_\eps(x_{b}))}{d_{\eps}(B_\eps(x_{a}),B_\eps(x_{b})) } e^{\frac{F(z)}{\eps}},
\end{align*}
where $d_{\eps}(B_\eps(x_{a}),B_\eps(x_{b}))$ is the geodesic distance between $B_\eps(x_{a})$ and $B_\eps(x_{a}),$ and $V_{\eps}(B_\eps(x_{a}),B_\eps(x_{b}))$ is the area of the 'smallest cross section'. This is in accordance with the classical result on parallel plate capacitors, where the capacity depends linearly on the area and is inversely proportional to their distance.

The statement \cref{eq:thm1-1}, when the saddle points are parallel,  means that each saddle point $z_1, \dots, z_N$ charges capacity and the total capacity is their sum. Again the situation is the same as in the case of parallel plate capacitors with capacity $C_1, \dots, C_N$, where the total capacity is the sum
\begin{align*}
  C =C_1 + \dots + C_N.
\end{align*}
On the other hand, if the plate capacitors are in series their total capacity  satisfies
\begin{align*}
  \frac{1}{C} = \frac{1}{C_1} + \dots + \frac{1}{C_N}
\end{align*}
which is precisely the statement in  \cref{eq:thm1-2}.

Using  the assumption \cref{eq:struc3} we  calculate in \cref{lem:geometric-d} and in \cref{lem:geometric-V}  more  explicit, but less geometric,  formulas  for the single saddle case in a domain $\Omega$. Namely, we have
\begin{align*}
  d_{\eps}(B_\eps(x_{a}),B_\eps(x_{b});\Omega) \simeq   e^{\frac{F(z)}{\eps}}\int_{\R} e^{-\frac{g_{z}(x_1)}{\eps}} \, d x_1
\end{align*}
and
\begin{align*}
  V_{\eps}(B_\eps(x_{a}),B_\eps(x_{b});\Omega)\simeq  e^{-\frac{F(z)}{\eps}} \int_{\R^{n-1}} e^{-\frac{G_{z}(x')}{\eps}} \, d x',
\end{align*}
and thus we recover the result in \cite{Berglund} .  In particular, if the saddle point is non-degenerate, i.e., $g_{z}$ and $G_{z}$ are second order polynomials   and the negative eigenvalue of $\nabla^2 F(z)$ is $-\lambda_1$, we may estimate
\begin{align*}
  d_{\eps}(B_\eps(x_{a}),B_\eps(x_{b})) \simeq \sqrt{\frac{2 \pi \eps}{\lambda_1}}\, e^{\frac{F(z)}{\eps}}
\end{align*}
and
\begin{align*}
  V_{\eps}(B_\eps(x_{a}), B_\eps(x_{b})) \simeq (2 \pi \eps)^{\frac{n-1}{2}} \frac{\sqrt{\lambda_1}\, e^{-\frac{F(z)}{\eps}}}{\sqrt{\det (\nabla^2 F(z))}}.
\end{align*}
In particular, we recover the classical formula \cref{e.kramer}.

Our second main theorem is an estimate on the so called metastable exit times. However, in order to state it we need some further definitions. Assume that the local minimas of $F$ are labelled $x_i$ and ordered such that $F(x_i) \leq F(x_j)$ if $i \leq j$. We will group the minimas at the same level using the sets $G_k$, $k=1,\ldots, K$, i.e. $x_i, x_j \in G_k$ if $F(x_i) = F(x_j)$, and $x \in G_i$ and $y \in G_j$, then $F(x) < F(y)$ for $i < j$. We also write $F(G_i) := F(x)$ with $x \in G_i$. Furthermore, we will denote $S_k = \bigcup_{i=1}^k G_i$ for $k=1,\ldots, K$.
We will also consider $G_k^\varepsilon = \bigcup_{x \in G_k} B_\varepsilon(x)$ and $S_k^\varepsilon = \bigcup_{i=1}^k G_i^\eps$.

In addition to the previous structural assumptions we assume further that for $\delta_2 \leq \delta_1$ small enough, it holds
\begin{equation} \label{eq:struc10}
  F(G_{k+1}) - F(G_k) \geq \delta_2
\end{equation}
for all $k=1,\ldots,K$. In our second theorem we give an upper bound on the exit time for the process defined by \cref{eq:SGD} to go from local minimum point in $G_{k+1}^\varepsilon $ to a lower one in $S_k$.

\begin{theorem}\label{thm2}
  Assume that $F$ satisfies the structural assumptions above, and let $\Omega$ be a domain that contains $S_{k+1}^\epsilon$. There exists an $0 < \eps_0 \leq \delta_2$ such that if $0 < \eps < \eps_0$, the following holds:

  For $x \in G_{k+1}^\varepsilon $ we have
  \begin{align*}
    \E^x[\tau_{S_k^\eps} \I_{\tau_{S_k^\eps} < \tau_{\Omega^c}}]
    \leq
    \frac{C e^{-F(G_{k+1})/\varepsilon} \sum_{x \in G_{k+1} } |O_{x,\varepsilon}|}{\max_{x \in G_k, y \in G_{k+1}^\varepsilon } \capacity(B_\varepsilon(x),B_\varepsilon(y);\Omega)}
    + C \varepsilon^{\alpha/2}.
  \end{align*}
  Let $x_a \in G_k$, $x_b \in G_{k+1} $ be a pair that maximizes the pairwise capacity. Then, with the notation of \cref{thm1}, we get in the parallel case
  \begin{align*}
    \E^x[\tau_{S_k^\eps} \I_{\tau_{S_k^\eps} < \tau_{\Omega^c}}]
    \leq
    \eps^{-1} \frac{C e^{-F(G_{k+1})/\varepsilon} \sum_{x \in G_{k+1} } |O_{x,\varepsilon}|}
    {\sum_{i=1}^N e^{-F(z_i)/\epsilon} \frac{\mathcal{H}^{n-1}(\{G_{z_i} < \epsilon\}) }{\mathcal{H}^{1}(\{g_{z_i} < \epsilon\})}}
    + C \varepsilon^{\alpha/2},
  \end{align*}
  and in the series case we get
  \begin{align*}
    \E^x[\tau_{S_k^\eps} \I_{\tau_{S_k^\eps} < \tau_{\Omega^c}}]
    \leq
    \frac{C}{\eps} \sum_{x \in G_{k+1} } \sum_{i=1}^N e^{(F(z_i)-F(G_{k+1}))/\varepsilon}
    \frac{|O_{x,\varepsilon}| \mathcal{H}^{1}(\{g_{z_i} < \epsilon\}) }{\mathcal{H}^{n-1}(\{G_{z_i} < \epsilon\})}
    + C \varepsilon^{\alpha/2}.
  \end{align*}
\end{theorem}

If both the minimas and saddles are non-degenerate points, and  there is only one saddle connecting $x_a,x_b$ (with $F(x_a) < F(x_b)$), where $x_a,x_b$ are the only minimas of $F$, the above estimate coincides with Eyring-Kramers formula (up to a constant)
\begin{align*}
  \E^x[\tau_{S_k^\eps} \I_{\tau_{S_k^\eps} < \tau_{\Omega^c}}] \leq C \frac{1}{\lambda_1} \sqrt{\frac{\det (\nabla^2 F(z))}{\det (\nabla^2 F(x_a))}} e^{(F(z)-F(x_a))/\epsilon}.
\end{align*}
Here $\lambda_1$ is the first eigenvalue of the Hessian of $F$ at the saddle $z$, and the additive error in \cref{thm2} can be removed for small $\epsilon$ as the right hand side of the above tends to $\infty$ as $\epsilon \to 0$.

\section{Preliminaries}% \label{sec:prelim}
The generator of the process \cref{eq:SGD} is the following elliptic operator
\begin{align} \label{eq:op}
  L_\varepsilon = -\varepsilon \Delta  + \nabla F \cdot \nabla.
\end{align}
In this section we study the potential and regularity theory associated with the operator \cref{eq:op}.
We provide the identities and pointwise estimates that we will need in the course of the proofs. Most of these are standard, but we provide them adapted to our situation for the reader's convenience. We note that in this section we only require that the potential $F$ is of class $C^2$ and satisfy the quadratic growth condition \cref{eq:struc0}.

\subsection{Potential theory}
\begin{definition}
  Let  $\Omega \subset \R^n$ be a regular domain and let $G_\Omega(x,y)$ be the Green's function for $\Omega$, i.e., for every $f \in C(\Omega)$ the function
  \begin{align*}
    u_f = \int_{\Omega} G_\Omega(x,y)f(y) dy
  \end{align*}
  is the solution of the Poisson equation
  \begin{align*}
    \begin{cases}
      L_\varepsilon u_f = f & \text{in } \,  \Omega \\
      u_f = 0 & \text{on } \,   \partial \Omega.
    \end{cases}
  \end{align*}
  The natural associated measures are the Gibbs measure $d\mu_\varepsilon = e^{-F/\varepsilon} dx$ and the Gibbs surface measure $d\sigma_\varepsilon = e^{-F/\varepsilon} d\mathcal{H}^{n-1}$.
\end{definition}

\begin{remark} \label{r:symmetry}
  Note that the Green's function is symmetric w.r.t. the Gibbs measure, i.e.
  \begin{align*}
    G(x,y) e^{-F(x)/\varepsilon} = G(y,x)e^{-F(y)/\varepsilon}.
  \end{align*}
\end{remark}
We also have the fundamental Green's identities. Here we assume that $\Omega$ is a Lipschitz domain and denote the inner normal by $n$.
\begin{lemma}%\label{lem:Greens}
  Let $\Omega$ be a smooth domain, $\psi,\phi$ be in $C^2(\overline \Omega)$, and $G_\Omega$ be the Green's function for $\Omega$. Then the following Green's identities holds {\bf (Green's first identity)}
  \begin{align} \label{eq:green1}
    \int_\Omega \psi L_\varepsilon \phi - \varepsilon \nabla \psi \cdot \nabla \phi d\mu_\varepsilon
    = \varepsilon \int_{\partial \Omega} \psi \nabla \phi \cdot n d\sigma_\varepsilon,
  \end{align}
  and {\bf (Green's second identity)}
  \begin{align}\label{eq:green2}
    \int_\Omega \psi L_\varepsilon \phi - \phi L_\varepsilon \psi d\mu_\varepsilon = \varepsilon \int_{\partial \Omega} \psi \nabla \phi \cdot n - \phi \nabla \psi \cdot n d\sigma_\varepsilon.
  \end{align}
  Furthermore, the following (balayage) representation formula holds: for every $g \in C(\partial \Omega)$ the function
  \begin{align} \label{eq:representation}
    u(x) = \varepsilon e^{F(x)/\varepsilon} \int_{\partial \Omega} g \nabla_y G_\Omega(y,x) \cdot n d\sigma_\varepsilon(y)
  \end{align}
  is the solution of the Dirichlet problem
  \begin{align*}
    \begin{cases}		L_\varepsilon u = 0 &\text{in } \,  \Omega \\
      u = g &\text{on } \,   \partial \Omega.
    \end{cases}
  \end{align*}
\end{lemma}
\begin{proof}
  Integration by parts gives
  \begin{align*}
    \int_\Omega \psi L_\varepsilon \phi d\mu_\varepsilon &= \int_\Omega \psi (-\varepsilon \Delta \phi + \nabla F \cdot \nabla \phi) d\mu_\varepsilon\\
    &= \int_\Omega (\varepsilon \nabla \psi \cdot \nabla \phi - \frac{\varepsilon}{\varepsilon} \psi \nabla F \cdot \nabla \phi + \psi \nabla F \cdot \nabla \phi) d\mu_\varepsilon + \varepsilon \int_{\partial \Omega} \psi \nabla \phi \cdot n d\sigma_\varepsilon \\
    &=\int_\Omega \varepsilon \nabla \psi \cdot \nabla \phi d\mu_\varepsilon + \varepsilon \int_{\partial \Omega} \psi \nabla \phi \cdot n d\sigma_\varepsilon.
  \end{align*}
  The second Green's identity follows from the first by applying it twice
  \begin{multline*}
    \int_\Omega \psi L_\varepsilon \phi - \varepsilon \nabla \psi \cdot \nabla \phi d\mu_\varepsilon
    -\int_\Omega \phi L_\varepsilon \psi + \varepsilon \nabla \psi \cdot \nabla \phi d\mu_\varepsilon\\
    = \varepsilon \int_{\partial \Omega} \psi \nabla \phi \cdot n d\sigma_\varepsilon-\varepsilon \int_{\partial \Omega} \phi \nabla \psi \cdot n d\sigma_\varepsilon.
  \end{multline*}
  We may now obtain the representation formula for the Dirichlet problem. We choose $\phi(x) = G_\Omega(x,y)$ and obtain by Green's second identity that
  \begin{align*}
    \psi(x)e^{-F(x)/\varepsilon} - \varepsilon \int_{\partial \Omega} \psi \nabla \phi \cdot n d\sigma_\varepsilon = \int_\Omega \phi L_\varepsilon \psi d\mu_\varepsilon.
  \end{align*}
  Now relabeling $x \to y$, we get the representation formula \cref{eq:representation}.
\end{proof}

Recall the definition of capacity (variational): for $A,B \subset \Omega$ two disjoint compact sets
\begin{align} \label{e:cap:precise}
  \capacity(A,B; \Omega) := \inf \left( \eps \int_{\Omega } |\nabla h|^2 e^{-\frac{F}{\eps}}\, dx \, : \,\, h \geq 1 \,\, \text{in } A, u \in H^1_0(\Omega \setminus B) \right).
\end{align}
The extension of capacity to open sets follows in the classical way
\begin{align*}
  \capacity(U,B;\Omega) := \sup\{\capacity(A,B;\Omega)\, :\, \text{$A$ compact and $A \subset U$}\}.
\end{align*}
It is well known that for bounded sets with regular boundary the continuity of the capacity implies that $\capacity(U,B;\Omega) = \capacity(\overline{U},B;\Omega)$. The extension w.r.t the second entry follows similarly.
The variational definition of the capacity has many equivalent forms, one that we will need is the one below:
%, which is sometimes called the balayage (sweeping) definition.

\begin{lemma} \label{l:capacity}
  Let $A,B \subset \Omega$ be two disjoint compact sets. Then the variational formulation of capacity coincides with the balayage definition, i.e.,
  \begin{align*} %\label{eq:balayage_capacity}
    \capacity(A,B;\Omega) = \sup \left \{\int_A e^{-F(y)/\varepsilon} d\mu(y): \supp \mu \subset A, \int_{\Omega} G_{\Omega \setminus B}(x,y) d\mu(y) \leq 1 \right \}.
  \end{align*}
  The unique measure which maximizes the above, i.e., satisfying
  \begin{align*}
    \int_A e^{-F(y)/\varepsilon} d\mu_{A,B}(y) = \capacity(A,B;\Omega), \quad \int_{\Omega} G_{\Omega \setminus B}(x,y) d\mu(y) \leq 1,
  \end{align*}
  is called the equilibrium measure $\mu_{A,B}$. The corresponding equilibrium potential is defined as $h_{A,B} = \int_{\Omega} G_{\Omega \setminus B} (x,y) d\mu_{A,B}(y)$ and is the minimizer of \cref{e:cap:precise}.

  If in addition $A,B$ are smooth, then we have
  \begin{equation} \label{eq:measure_energy_equality}
    \begin{split}
      \capacity(A,B;\Omega)
      &=
      \int_A e^{-F(y)/\varepsilon} d\mu_{A,B}(y)
      =
      \varepsilon \int |\nabla h_{A,B}|^2 d\mu_\varepsilon
      \\
      &= -\varepsilon \int_{\partial A} \nabla h_{A,B} \cdot n d\sigma_\varepsilon.
    \end{split}
  \end{equation}
\end{lemma}
\begin{proof}
  The claim follows from the symmetry of the Green's function, \cref{r:symmetry}, and the strong maximum principle that $h_{A,B} = 1$ in $A$, see \cite{AE,C}.
  From \cref{eq:green1} we get that
  \begin{align*}
    \int_{\Omega \setminus B} h_{A,B} L_\varepsilon h_{A,B} - \varepsilon |\nabla h_{A,B}|^2 d\mu_\varepsilon
    = \varepsilon \int_{\partial (\Omega \setminus B)} h_{A,B} \nabla h_{A,B} \cdot n d\sigma_\varepsilon.
  \end{align*}
  Using $h_{A,B} = 0$ on $\partial (\Omega \setminus B)$ we see that the right hand side of the above is zero.  Moreover, from  $L_\epsilon h_{A,B} = \mu_{A,B}$ and from the definition of $d\mu_\epsilon$ we get
  \begin{align} \label{eq:measure_energy}
    \int_A e^{-F(y)/\varepsilon} d\mu_{A,B}(y) = \varepsilon \int_{\Omega} |\nabla h_{A,B}|^2 d\mu_\varepsilon.
  \end{align}
  Note that since $h_{A,B} = 1$ in $A$ and $0$ on $\partial (\Omega \setminus B)$, and since $L_\epsilon h_{A,B} = 0$ in $\Omega \setminus (A \cup B)$, we have by the uniqueness of the solution to the Dirichlet problem that $h_{A,B}$ coincides with the variational minimizer of \cref{e:cap:precise}. This establishes the first two equalities of \cref{eq:measure_energy_equality}.

  To prove the last equality in \cref{eq:measure_energy_equality} we insert $h_{A,B} = \phi = \psi$ into \cref{eq:green1} (Green's first identity) and get
  \begin{align*}
    \int_{\Omega \setminus (A \cup B)} h_{A,B} L_\varepsilon h_{A,B} - \varepsilon |\nabla h_{A,B}|^2 d\mu_\varepsilon = \varepsilon \int_{\partial (\Omega \setminus (A \cup B))} h_{A,B} \nabla h_{A,B} \cdot n d\sigma_\varepsilon.
  \end{align*}
  Since $\supp \mu_{A,B} \subset A$, and $h_{A,B} = 0$ on $B$ and $1$ on $A$, we get
  \begin{align} \label{eq:energy_normal_integral}
    - \varepsilon \int_\Omega |\nabla h_{A,B}|^2 d\mu_\varepsilon = \varepsilon \int_{\partial A} \nabla h_{A,B} \cdot n d\sigma_\varepsilon,
  \end{align}
  where $n$ is the outward unit normal of $A$.
  The result now follows from \cref{eq:measure_energy,eq:energy_normal_integral}.
\end{proof}

\begin{definition} \label{d:POEP}
  Let $\Omega$ be a smooth domain and $A \subset \Omega$. Then we define the potential of the equilibrium potential as
  \begin{align*}%\label{eq:wa}
    w_{A,\Omega}(x) = \int_\Omega G_{\Omega \setminus A}(x,y) h_{A,\Omega^c}(y) dy.
  \end{align*}
\end{definition}
The definition of the potential of the equilibrium potential might seem technical at first. However, $w_{A,\Omega}$ has a clear probabilistic interpretation as the expected hitting time of hitting $A$ of a process killed at $\partial \Omega$. Indeed, the probabilistic interpretation of $h_{A,\Omega^c}$ is $\P(\tau_A < \tau_{\Omega^c})$ i.e. the probability of hitting $A$ before ${\Omega^c}$.
By Dynkin's formula we see that then
\begin{align*}
  w_{A,{\Omega}}(x)
  &=
  \E^x[w_{A,{\Omega}}(X_{\tau_{A \cup {\Omega^c}}})] - \E^x \left [ \int_0^{\tau_{A \cup {\Omega^c}}} L_\eps w_{A,{\Omega}}(X_t) dt \right ]
  \\
  &=
  \int_0^{\infty} \E^x \left [ \I_{t \leq \tau_{A \cup {\Omega^c}}} \E^{X_t}[\I_{A}(X_{\tau_{A \cup {\Omega^c}}})] \right ]dt
  \\
  &=
  \int_0^{\infty} \E^x \left [ \I_{t \leq \tau_{A \cup {\Omega^c}}} \I_{A}(X_{\tau_{A \cup {\Omega^c}}}) \right ]dt
  \\
  &=
  \E^x[\tau_A \I_{\tau_A < \tau_{\Omega^c}}].
\end{align*}

We also have the following integration by parts formula for the potential of the equilibrium potential:

\begin{lemma}\label{lem:POEP}
  Let $\Omega$ be a smooth domain, let $A \subset \Omega$ be a smooth set, and assume that $B_{2\rho}(x) \subset \Omega \setminus A$. Then
  \begin{align*}%\label{eq:main_point}
    \int_{\partial B_\rho(x)} w_{A,\Omega}(y) e^{-F(y)/\varepsilon} d\mu_{B_\rho(x),A}(y) = \int_{\Omega \setminus A} h_{A,\Omega^c}(z) h_{B_\rho(x),A}(z) d\mu_\varepsilon.
  \end{align*}
\end{lemma}
The above statement looks more familiar if we write it in the formal way as
\begin{align*}
  \int w_{A,\Omega} e^{-F/\varepsilon} L_\eps h_{B_\rho,A} = \int L_\varepsilon w_{A,\Omega} e^{-F/\varepsilon} h_{B_\rho,A} dz.
\end{align*}
\begin{proof}
  Using the definition of $w_{A,\Omega}$ and Fubini's theorem
  \begin{align*}
    \int_{\partial B_\rho(x)} &w_{A,\Omega}(y) e^{-F(y)/\varepsilon} d\mu_{B_\rho(x),A}(y) \\
    &=
    \int_{\partial B_\rho(x)} \left (\int_{\Omega \setminus A} G_{\Omega \setminus A}(y,z) h_{A,\Omega^c}(z)dz \right ) e^{-F(y)/\varepsilon} d\mu_{B_\rho(x),A}(y) \\
    &=
    \int_{\Omega \setminus A} h_{A,\Omega^c}(z) \int_{\partial B_\rho(x)} G_{\Omega \setminus A}(y,z) e^{-F(y)/\varepsilon} d\mu_{B_\rho(x),A}(y) dz.
  \end{align*}
  Using the symmetry of the Green's function, \cref{r:symmetry},
  % To continue we need to switch places with $y,z$ in the Green's function, which satisfies
  % \begin{align*}
  % 	G(x,y) e^{-F(x)/\varepsilon} = G(y,x)e^{-F(y)/\varepsilon}
  % \end{align*}
  % by Lemma \ref{r:symmetry}.
  \begin{align*}
    \int_{\Omega \setminus A} &h_{A,\Omega^c}(z) \int_{\partial B_\rho(x)} G_{\Omega \setminus A}(y,z) e^{-F(y)/\varepsilon} d\mu_{B_\rho(x),A}(y) dz \\
    &=
    \int_{\Omega \setminus A} h_{A,\Omega^c}(z) e^{-F(z)/\varepsilon} \int_{\partial B_\rho(x)} G_{\Omega \setminus A}(z,y) d\mu_{B_\rho(x),A}(y) dz.
  \end{align*}
  Note that $\supp \mu_{B_\rho(x),A} \subset \partial B_\rho(x)$ and as such
  \begin{align*}
    u (z) = \int_{\Omega} G_{\Omega \setminus A}(z,y) d\mu_{B_\rho(x),A}(y) = \int_{\partial B_\rho(x)} G_{\Omega \setminus A}(z,y) d\mu_{B_\rho(x),A}(y)
  \end{align*}
  solves the equation $L_\epsilon u = \mu_{B_\rho(x),A}$ in $\Omega$ and $u = 0$ on $\partial {\Omega \setminus A}$. Consequently, by the uniqueness result to the Dirichlet-Poisson problem, we get $u(z) = h_{B_\rho(x),A}(z)$. Hence
  \begin{multline*}
    \int_{\Omega \setminus A} h_{A,\Omega^c}(z) e^{-F(z)/\varepsilon} \int_{\partial B_\rho(x)} G_{\Omega \setminus A}(z,y) d\mu_{B_\rho(x),A}(y) dz
    \\
    =
    \int_{\Omega \setminus A} h_{A,\Omega^c}(z) h_{B_\rho(x),A}(z) d\mu_\varepsilon(z).
  \end{multline*}
  Combining the equalities above yields the result.
\end{proof}

\subsection{Classical pointwise estimates}

In this section we recall classical pointwise estimates for functions which satisfy
\begin{align*}
  L_\varepsilon u = f
\end{align*}
in a domain $\Omega$, where the operator $L_\eps$ is defined in \cref{eq:op}. First, since we assume that $F \in C^2(\R^n)$, then for all H\"older continuous $f$ the solutions of the above equation are $C^{2,\alpha}$-regular, see \cite{GT}. However, these regularity estimates depend on $\eps$ and blow up as $\eps \to 0$. The point is that we may obtain regularity estimates for constants independent of $\eps$ if we restrict ourselves on small enough scales. To this aim, for a given domain  $\Omega$ we choose a positive number $\nu$ such that
\begin{align*}
  \frac{\|\nabla F\|_{L^\infty(\Omega)}}{\varepsilon}  \leq \nu.
\end{align*}
We have the following two theorems from \cite{GT}.

\begin{lemma}\label{l:harnack}
  Let $\Omega$ be a domain and let $u \in C^2(\Omega)$ be a non-negative function satisfying $L_\varepsilon u = 0$. Then for any $B_{3R}(x) \subset \Omega$ it holds that
  \begin{align*}
    \sup_{B_R(x)} u \leq C \inf_{B_R(x)} u
  \end{align*}
  for a  constant $C= C(n,\nu R)$. In particular,  if $\|\nabla F\|_{L^\infty(\Omega)} \leq L$, then for $R \leq \frac{\varepsilon}{L}$ the constant $C$ is independent of 		$\eps$. 		Furthermore, for $p \in (1,\infty)$ and any number $k$ we have
  \begin{align*}
    \sup_{B_R(x)} |u-k| \leq C_p \left ( \fint_{B_{2R}(x)} |u-k|^p dx\right )^{1/p}
  \end{align*}
  and
  \begin{align*}
    \left ( \fint_{B_{2R}(x)} |u|^p dx\right )^{1/p} \leq C_p  \inf_{B_R(x)} u,
  \end{align*}
  where the symbol $\fint$ denotes the average integral, and the constant $C_p$ in addition to above depends also on $p$.
\end{lemma}
In the non-homogeneous case $L_\varepsilon u = f$ we have the following generalization of Harnack's inequality.
\begin{lemma} \label{l:maximum_principle}
  Let $\Omega$ be a domain and let $u \in C^2(\Omega)$ be a non-negative function satisfying $L_\varepsilon u = f$. Then for any $B_{3R}(x) \subset \Omega$ it holds that
  \begin{align*}
    \sup_{B_R(x)} u \leq C \left ( \inf_{B_R(x)} u + \frac{R}{\varepsilon} \|f\|_{L^n(B_{2R}(x))} \right )
  \end{align*}
  for a constant $C= C(n,\nu R)$. In particular,  if $\|\nabla F\|_{L^\infty(\Omega)} \leq L$, then for $R \leq \frac{\varepsilon}{L}$ the constant $C$ is independent of 		$\eps$  and we have
  \begin{align*}
    \sup_{B_R(x)} u \leq C \left (\inf_{B_R(x)} u + \|f\|_{L^n(B_{2R}(x))} \right ).
  \end{align*}
\end{lemma}

The Harnack inequality in \cref{l:harnack} holds also in the case of the punctured ball.

\begin{lemma}\label{lem:punctured_harnack}
  Let $u \in C^2(B_{3R}(x) \setminus \{x\})$ be a non-negative function satisfying $L_\varepsilon u = 0$ in $B_{3R}(x) \setminus \{x\}$. Then
  \begin{align*}
    \sup_{\partial B_R(x)} u \leq C \inf_{\partial B_R(x)} u
  \end{align*}
  for a  constant $C= C(n,\nu R)$.
\end{lemma}
\begin{proof}
  By translating the coordinates we may assume that $x = 0$. Let $x_0, y_0 \in \partial B_R$ be such that $\sup_{\partial B_R(x)} u = u(x_0)$ and $\inf_{\partial B_R(x)} u = u(y_0)$. We choose points $x_1, \dots, x_{N-1}, x_N \in \partial B_R$ such that $|x_i-x_{i-1}| \leq R/4$ and $x_N = y_0$. Note that the number $N$ is bounded.  Now we may use  Harnack's inequality \cref{l:harnack} in balls $B_{R/4}(x_i)$ to get
  \begin{align*}
    u(x_{i-1}) \leq C u(x_{i}).
  \end{align*}
  We obtain the claim by summing over $i=1, \dots, N$.
\end{proof}
\begin{lemma} \label{lem:harnack+}
  Let $\Omega$ be a domain and let $u,h \in C^2(\Omega)$ be non-negative functions such that $L_\eps u = h$ and  $h$ satisfies Harnack's inequality with constant $c_0$. Then the function $v = u+h$ satisfies Harnack's inequality, i.e.,   for all $B_{3R}(x) \subset \Omega$ it holds that
  \begin{align*}
    \sup_{B_R(x)} v \leq C \inf_{B_R(x)} v
  \end{align*}
  for a constant  $C = C(n,\nu R, R^2/\varepsilon,c_0)$. In particular, if $\|\nabla F\|_{L^\infty(\Omega)} \leq L$,  then for $R \leq \min\{ \varepsilon/L, \sqrt{\eps}\}$ the constant $C$ is independent of $\eps$.
\end{lemma}
\begin{proof}
  Again we may assume that $x = 0$. Using \cref{l:maximum_principle} and Harnack's inequality for $h$ yields
  \begin{align*}
    \sup_{B_R} u &\leq C \inf_{B_R} u + \frac{C R}{\varepsilon} \|h\|_{L^n(B_{2R})} \\
    &\leq C \inf_{B_R} u + \frac{C R}{\varepsilon} |B_R|^{1/n} \inf_{B_{R}}  h\\
    &\leq C \inf_{B_R} u + \frac{C R^2}{\varepsilon} \inf_{B_{R}}  h.
  \end{align*}
  Now, using Harnack's inequality for $h$ again, we obtain
  \begin{align*}
    \sup_{B_R} v &\leq \sup_{B_R} u + \sup_{B_R} h \leq C \inf_{B_R} u + C \inf_{B_R} h\leq  C \inf_{B_R} v
  \end{align*}
  for a constant $C$ as in  the statement. This proves the claim.
\end{proof}

The Harnack's inequality in \cref{l:harnack}  implies H\"older continuity for solutions of $L_\eps u = 0$.

\begin{lemma}\label{lem:holder}
  Let $u \in C^2(B_{3R}(x))$ be a  function such that for any constant $c$, for which $v = u+c$ is non-negative, the function $v$  satisfies Harnack's inequality  with constant $C_0$, independent of $c$. Then there exists $C = C(C_0) > 1$ and $\alpha = \alpha(C_0) \in (0,1)$ such that, for all  $\rho \leq R$, it holds that
  \begin{align*}
    \osc_{B_\rho(x)} u \leq C \left ( \frac{\rho}{R}\right )^\alpha \osc_{B_{R}(x)} u.
  \end{align*}
  In particular,  if  $u,h \in C^2(\Omega)$ are non-negative functions such that $L_\eps u = h$ and  $h$ satisfies Harnack's inequality with constant $C_0$, then $u+h$ satisfies the 			estimate above.
\end{lemma}

\begin{proof}
  The proof follows verbatim from the classical proof of Moser, see \cite[Theorem 8.22]{GT}.
\end{proof}

\section{Technical lemmas} \label{sec:technical}

In this section we provide some preliminary results for the proofs of the main theorems. We recall that we assume that the potential $F$ satisfies the structural assumptions from
\cref{sec:ass}, and that from this moment on our constants are allowed to depend on the data, see paragraph after \cref{e.etahat}.

\subsection{Rough estimates for potentials}
In this subsection we provide estimates for the capacitary potential $h_{A,B}$, when $A$ and $B$ are two disjoint closed sets.
The first estimate is the so called renewal estimate of \cite{BGK}. In order to trace dependencies of constants, we provide a proof.
\begin{lemma} \label{lem:replacement}
  Let $\Omega$ be a smooth domain, let $A, B \subset \Omega$ be disjoint smooth sets, and consider $h_{A,B}$ as the capacitary potential in $\Omega$. Assume that $B_{4\varrho}(x) \subset (\Omega \setminus (A \cup B))$, and that $r \leq \min\left \{ \frac{\varepsilon}{\|\nabla F\|_{L^\infty(B_{2\varrho}(x))}}, \varrho \right \}$.
  Then there exists a constant $C = C(n,\nu) > 1$ such that
  \begin{align*}
    h_{A,B}(x) &\leq C
    \frac{\capacity(B_{r}(x),A;\Omega)}{\capacity(B_{r}(x),B;\Omega)}.
  \end{align*}
\end{lemma}
\begin{proof}
  Again, without loss of generality, we may assume that $x = 0$. Since $h_{A,B \cup B_r} = h_{A,B}$ on $\partial (\Omega \setminus (A \cup B))$, we can use \cref{eq:representation} to represent $h_{A,B}$ as follows
  \begin{align} \label{e.hab.1}
    h_{A,B}(z) = \varepsilon e^{F(z)/\varepsilon} \int_{\partial (\Omega \setminus (A \cup B))} h_{A,B \cup B_r} \nabla G_{\Omega \setminus (A \cup B)}(y,z) \cdot n d\sigma_\varepsilon(y).
  \end{align}
  Now by Green's second identity \cref{eq:green2} in $\Omega \setminus (A \cup B \cup B_r)$ and \cref{e.hab.1} we see that, for $z \in \Omega$,
  \begin{multline} \label{e.hab.2}
    h_{A,B}(z) = h_{A,B \cup B_r}(z) \\
    - \varepsilon e^{F(z)/\varepsilon} \int_{\partial B_r} G_{\Omega \setminus (A \cup B)}(y,z) \nabla h_{A,B \cup B_r}(y) \cdot n d\sigma_\varepsilon(y),
  \end{multline}
  where $n$ is the inward unit normal of $B_r$.
  First note that by \cref{eq:green2} we can identify the equilibrium measure as
  \begin{align*}
    \mu_{B \cup B_r,A} = -\varepsilon \nabla h_{B \cup B_r,A} \cdot n d\sigma_\varepsilon =
    \varepsilon \nabla h_{A,B \cup B_r} \cdot n d\sigma_\varepsilon.
  \end{align*}
  Using that $h_{A,B\cup B_r} = 1-h_{B \cup B_r, A}$, together with the above and \cref{e.hab.2}, we get for $z \in \overline{B_r}$ (since $h_{A, B \cup B_r}(z) = 0$) that
  % If $z \in B_r$, then $h_{A, B \cup B_r}(z) = 0$, using this, $h_{A,B\cup B_r} = 1-h_{B \cup B_r, A}$, \cref{eq:measure_energy_equality}, and the above, we get for $z \in B_r$
  \begin{align}\label{e.hab.3}
    h_{A,B}(z) &= \int_{\partial B_r} G_{\Omega \setminus (A \cup B)}(y,z) e^{(F(z)-F(y))/\varepsilon} d\mu_{B \cup B_r,A}(y).
  \end{align}
  % This immediately implies that
  % % We can now bound $h_{A,B}$ from above in $C$ by simply taking out the supremum of the Green's function and using \cref{eq:balayage_capacity}
  % \begin{align} \label{e.hab.3}
  % 	h_{A,B}(z) \leq \sup_{y \in \partial B_r} G_{\Omega \setminus (A \cup B)}(y,z) e^{F(z)/\varepsilon} \int_{\partial B_r} e^{-F(y)/\varepsilon} d\mu_{B \cup B_r,A}(y).
  % \end{align}
  % $\phi = h_{B\cup B_r, A}$, $\psi = G_{\Omega \setminus (A \cup B)}$
  % \begin{align*}
  % 	\int_{\Omega \setminus (A \cup B \cup B_r)} G_{\Omega \setminus (A \cup B)} L_\varepsilon h_{B\cup B_r, A} - h_{B\cup B_r, A} L_\varepsilon G_{\Omega \setminus (A \cup B)} d\mu_\varepsilon
  % 	\\
  % 	=
  % 	\varepsilon \int_{\partial (\Omega \setminus (A \cup B \cup B_r))} G_{\Omega \setminus (A \cup B)} \nabla h_{B\cup B_r, A} \cdot n - h_{B\cup B_r, A} \nabla G_{\Omega \setminus (A \cup B)} \cdot n d\sigma_\varepsilon.
  % \end{align*}

  First note that $\mu_{B_r \cup B,A} \lvert_{\partial B_r}$ is an admissible measure for $\capacity(B_r,A;\Omega)$, which follows from the fact that by the comparison principle, the potentials for ordered measures are ordered and the support of $\mu_{B_r \cup B,A} \lvert_{\partial B_r}$ is in $\overline{B_r}$.
  To bound $h_{A,B}$ from above, note that by the balayage representation of capacity (see \cref{l:capacity}) and the above, we obtain
  \begin{align*}
    \int_{\partial B_r} e^{-F(y)/\varepsilon} d\mu_{B \cup B_r,A}(y) \leq \capacity(B_r,A;\Omega).
  \end{align*}
  Applying the above to \cref{e.hab.3} gives, for $z \in B_r$,
  \begin{align}\label{e:habcap}
    h_{A,B}(z) \leq \sup_{y \in \partial B_r} G_{\Omega \setminus (A \cup B)}(y,z) e^{F(z)/\varepsilon}\capacity(B_r,A;\Omega).
  \end{align}
  It remains to bound the Green's function.
  For $z \in \overline{B_r}$ we have by \cref{eq:representation,r:symmetry,eq:measure_energy_equality} that
  \begin{equation} \label{e:greenlower}
    \begin{aligned}
      1 &= h_{B_r,A \cup B}(z) = \int_{\partial B_r} G_{\Omega \setminus (A \cup B)}(z,y)  d\mu_{B_r,A \cup B}(y) \\
      &= \int_{\partial B_r} G_{\Omega \setminus (A \cup B)}(y,z) e^{(F(z) -F(y))/\varepsilon} d\mu_{B_r,A \cup B}(y) \\
      &\geq \inf_{\partial B_r} G_{\Omega \setminus (A \cup B)}(y,z) e^{F(z)/\varepsilon} \capacity(B_r,A \cup B;\Omega)
      \\
      &\geq \inf_{\partial B_r} G_{\Omega \setminus (A \cup B)}(y,z) e^{F(z)/\varepsilon} \capacity(B_r,B;\Omega).
    \end{aligned}
  \end{equation}
  %
  %
  % Note that by using the representation formula for $h_{B \cup B_r, A}$ in $B_r$ and using \cref{r:symmetry} we get
  % \begin{align} \label{e:greenlower}
  % 	\inf_{\partial B_r} G_{\Omega \setminus (A \cup B)}(y,x) e^{F(x)/\varepsilon} \leq \frac{1}{\capacity(B_r,A \cup B;\Omega)}.
  % \end{align}
  Now putting together \cref{e:habcap,e:greenlower,lem:punctured_harnack} we are done.
\end{proof}

The result below is a version of the rough capacity bound of \cite{BGK}, but we give a simplified proof. We will later use a similar argument in the proof of \cref{thm1}.

\begin{lemma} \label{l:capacity:rough}
  Let $D \subset B_R$ be a smooth closed set. Let $x \in B_R \setminus D$ be such that $B_{4\rho}(x) \subset B_R \setminus D$, for $\rho \leq \varepsilon$. Then there exists constants $q_1,q_2 \in \R$ and $C > 1$ such that
  \begin{align*}
    \frac{1}{C}\varepsilon^{q_1} \rho^{n-1} e^{-F(x;D)/\varepsilon} \leq \capacity{B_\rho(x)}{D} \leq C \varepsilon \rho^{q_2} e^{-F(x;D)/\varepsilon}.
  \end{align*}
\end{lemma}
\begin{proof}
  We assume without loss of generality that $F(B_\rho;D) = 0$, since the quantities can always be scaled back.
  Consider $\gamma \in \mathcal{C}(B_\rho(x),D;B_R)$ (i.e. a curve connecting $B_\rho(x)$ and $D$ inside $\Omega$) such that $\sup_t F(\gamma(t)) \leq C\eps$ and let $u(z) = h_{D,B_\rho(x)}(z)$. We first note by \cref{l:capacity} that
  \begin{align*}
    \capacity(B_\rho(x),D) = \varepsilon \int |\nabla u|^2 e^{-F(y)/\varepsilon} dy.
  \end{align*}
  Fix an $n-1$ dimensional disk $D_\rho$ of radius $\rho$. Then by Cauchy-Schwarz
  \begin{align*}
    \int_{\Omega} |\nabla u|^2 e^{-F(y)/\varepsilon} dy
    \geq \int_{0}^1 \int_{D_\rho} \left |\left \langle \frac{\dot \gamma}{|\dot \gamma|} , \nabla u(\gamma(t)+z) \right \rangle \right |^2 |\dot \gamma| d\sigma_\varepsilon(z)dt.
  \end{align*}

  By the fundamental theorem of calculus and Cauchy-Schwarz, we have for a fixed point $z \in D_\rho$ that
  \begin{align*}
    1
    &=
    u(\gamma(1)) - u(\gamma(0))
    =
    \int_0^1 \frac{d}{dt} u(\gamma(t)+z)dt
    \\
    &=
    \int_0^1 \frac{d}{dt} u(\gamma(t)+z)\frac{\sqrt{|\dot \gamma|}}{\sqrt{|\dot \gamma|}} e^{-F(\gamma(t)+z)/(2\varepsilon)}e^{F(\gamma(t)+z)/(2\varepsilon)}dt
    \\
    &\leq
    \left (\int_0^1 |\frac{d}{dt} u(\gamma(t)+z)|^2 \frac{1}{|\dot \gamma|}e^{-F(\gamma(t)+z)/\varepsilon} dt \right )^{1/2} \left (\int |\dot \gamma| e^{F(\gamma(t)+z)/\varepsilon}dt \right )^{1/2}.
  \end{align*}
  From the above we get
  \begin{align*}
    \int_{\Omega} |\nabla u|^2 e^{-F(y)/\varepsilon} dy
    &\geq
    \int_{0}^1 \int_{D_\rho} \left |\left \langle \frac{\dot \gamma}{|\dot \gamma|} , \nabla u(\gamma(t)+z) \right \rangle \right |^2 |\dot \gamma| d\sigma_\varepsilon(z)dt
    \\
    &\geq
    \int_{D_\rho} \left ( \int_0^1 |\dot \gamma| e^{F(\gamma(t)+z)/\varepsilon} dt \right )^{-1}  d\sigma_\varepsilon(z).
  \end{align*}
  Now since $F$ is Lipschitz in $B_R$ and $F(B_\rho;D) = 0,$ we know that there exists a constant $C(\gamma)$ such that, for $z \in D_\rho$ and $\rho < 2\varepsilon$,
  \begin{align} \label{e.gammaeF}
    \int_0^1 |\dot \gamma| e^{F(\gamma(t)+z)/\varepsilon} dt \leq C(\gamma).
  \end{align}
  In the above the constant $C$ depends on the length of $\gamma$, which can be assumed to be bounded. To see this, take an $\epsilon$ neighborhood of $\gamma$, $E_\eps$ and consider a set of balls $\bigcup_i B_\eps(y_i) \supset E_\eps$ such that $\bigcup_i B_{\eps}(y_i) \subset E_{C_1 \eps}$ (for some large $C_1$), the maximal number of such balls needed is $C_1 R^n/\epsilon^n$. If we construct a piecewise linear curve $\gamma_\eps$ connecting the center of each ball in the covering, this curve will be inside $E_{C_1\eps}$ and its length will be bounded by $2 C_1 R^n/\eps^{n-1}$.
  This newly constructed curve can be mollified to achieve a smooth curve without increasing the length by more than a factor. From the above and the Lipschitz continuity of $F$ it is clear that $\sup_t F(\gamma_\eps(t)) \leq C\eps$, and as such we can replace $\gamma$ with $\gamma_\eps$ in the above and get from \cref{e.gammaeF} that there is a constant $C > 1$ depending only on the data such that
  \begin{align*}
    \int_0^1 |\dot \gamma| e^{F(\gamma(t)+z)/\varepsilon} dt \leq \eps^{1-n} C.
  \end{align*}
  This implies that for a new constant $C$ we have
  \begin{align*}
    \int_{\Omega} |\nabla u|^2 e^{-F(y)/\varepsilon} dy \geq C \eps^{n-1} \rho^{n-1},
  \end{align*}
  which completes the proof of the lower bound after rescaling our potential $F$.

  To prove the upper bound we have two possible cases:
  In the case when $F(x;D) = F(x)$ we can take a cutoff function $\chi_{B_\rho(x)} \leq \phi \leq \chi_{B_{2\rho}(x)}$ where $|\nabla \phi| \leq C/\rho$ as a competitor in the variational formulation of capacity \cref{e:cap:precise}. Then
  \begin{align*}
    \int_{\Omega} |\nabla \phi|^2 e^{-F(y)/\varepsilon} dx
    =
    \int_{B_{2\rho}(x)} |\nabla \phi|^2 e^{-F(y)/\varepsilon} dx
    \leq
    C \rho^{n-2}.
  \end{align*}
  In the case where $F(x;D) > F(x)$, consider the set $\hat D = \{z \in B_R: F(z) \leq F(x;D)\}$ and let $\hat D_1$ be the component that intersects $D$. We set $\tilde D = (\hat D_1 \cup D) \setminus B_{4\rho}(x)$. By the Lipschitz continuity, we know that $\inf_{\tilde D} F > - C\rho$. We take $\chi_{\tilde D + B_{\rho}} \leq \phi \leq \chi_{\tilde D + B_{2\rho}}$, where $|\nabla \phi| \leq C/\rho$, and get
  \begin{align*}
    \int_{\Omega} |\nabla \phi|^2 e^{-F(y)/\varepsilon} dx
    &=
    \int_{(\tilde D + B_{2\rho}) \setminus \tilde D} |\nabla \phi|^2 e^{-F(y)/\varepsilon} dx
    \\
    &\leq
    C |(\tilde D + B_{2\rho}) \setminus \tilde D|.
  \end{align*}
  Again, the upper bound follows from rescaling the potential $F$ as in the case of the lower bound. This completes the whole proof.
\end{proof}

\begin{lemma} \label{l:hab:bound}
  Let $A,B \subset B_{R}$ be smooth disjoint sets, and let $x \in B_R$ be such that $B_\varepsilon(x) \subset B_R \setminus (A \cup B)$, $\varepsilon \in (0,1)$. Then, if $F(x;B) \leq F(x; A)$, there exists constants $q$ and $C$ such that
  \begin{align*}
    h_{A,B}(x) \leq C \varepsilon^q e^{-(F(x;A)-F(x;B))/\varepsilon}.
  \end{align*}
\end{lemma}
\begin{proof}
  Let $L:=\|\nabla F\|_{L^\infty(B_{R})}$. By combining \cref{lem:replacement,l:capacity:rough} with $R = \varepsilon$, $r = \min\{\varepsilon/L,\varepsilon\}$ yields the result.
\end{proof}
\begin{remark}
  By relabeling $A,B$ to $B,A$ and using the fact that $h_{A,B} = 1-h_{B,A}$, we get that if the reverse inequality holds, i.e. $F(x;B) > F(x; A)$, then
  \begin{align*}
    1-h_{A,B}(x) \leq C \varepsilon^q e^{-(F(x;B)-F(x;A))/\varepsilon}.
  \end{align*}
\end{remark}

\begin{lemma} \label{l:hab:levelset:rough}
  Let $\Omega$ be a smooth domain and let $x_a,x_b \in \Omega \subset B_R$ be two local minimum points of $F$. Fix $0 < \delta < \delta_1$ and assume that $U_{-\delta/3} = \{ x: F(x) < F(x_a; x_b) - \delta/3\} \subset \Omega$.
  Then there exists an $\varepsilon_0 \in (0,1)$ and a constant $C = C > 1$ such that, for any $0 \leq \varepsilon \leq \varepsilon_0$ for which $B_{3 \varepsilon}(x_a),B_{3 \varepsilon}(x_b ) \subset U_{-\delta/3}$, the following holds:
  If $U_i$ is a component of $U_{-\delta/3}$, then
  \begin{align*}
    \underset{U_i}{\osc\,} h_{B_\varepsilon(x_a),B_\varepsilon(x_b)} \leq C\varepsilon.
  \end{align*}
\end{lemma}

\begin{proof}
  Consider any component $U_i$ of $U_{-\delta/3}$. We note that we can take $\varepsilon$ small enough depending on the Lipschitz constant of $F$ in $B_R$ and $\delta$ such that there exists a Lipschitz domain $D_i$ satisfying
  \begin{align*}
    U_i + B_\varepsilon \subset D_i \subset  U_{-\delta/4}.
  \end{align*}
  For simplicity, denote $u:=h_{B_\varepsilon(x_a),B_\varepsilon(x_b)}$. Since $D_i$ is Lipschitz we may use the Poincaré inequality to get
  \begin{align*}
    \int_{D_{i}} |u - u_{D_i}|^2 dx \leq C\int_{D_{i}} |\nabla u|^2 dx.
  \end{align*}
  Using that $D_i \subset  U_{-\delta/4}$ together with \cref{l:capacity}
  \begin{align*}
    \int_{D_{i}} |\nabla u|^2 dx
    &\leq
    e^{\sup_{D_{i}} F/\varepsilon} \int_{D_{i}} |\nabla u|^2 e^{-F(x)/\varepsilon}dx
    \\
    &\leq
    \varepsilon^{-1} e^{\sup_{D_{i}} F/\varepsilon} \capacity(B_\varepsilon(x_a),B_\varepsilon(x_b);\Omega).
  \end{align*}
  Using the definition of $U_{-\delta/4}$ and \cref{l:capacity:rough}, we get
  \begin{align} \label{e.small.l2}
    \int_{D_{i}} |u - u_{D_i}|^2 dx \leq C \varepsilon^{q_1} e^{-\delta/4\varepsilon}
  \end{align}
  for some constant $q_1 \in \R$. Now, for any $x_0 \in U_i$ we have by \cref{l:harnack} that
  \begin{align*}
    \sup_{B_\varepsilon(x_0)} |u-u_{D_i}|^2 &\leq C \left ( \fint_{B_{2\varepsilon}} |u-u_{D_i}|^2 dx\right )
  \end{align*}
  which together with \cref{e.small.l2} gives
  \begin{align*}
    \sup_{B_\varepsilon(x_0)} |u-u_{D_i}|^2 \leq C\varepsilon^{q_1-n} e^{-\delta/4\varepsilon}.
  \end{align*}
  Since $x_0$ was an arbitrary point in $U_i$ we conclude that there exists  $\varepsilon_0 \in (0,1)$ depending only on the data such that if $\varepsilon < \varepsilon_0$, the claim holds.
\end{proof}

We conclude this subsection with an estimate relating the value of the potential of the equilibrium potential to the ratio of the $L^1$ norm of the equilibrium potential and the capacity.

\begin{lemma} \label{lem:POEP_ratio}
  Let $\Omega$ be a smooth domain and let $A \Subset \Omega$ be a smooth open set and consider $w_{A,\Omega}$ as the potential of the equilibrium potential in $\Omega$ (see \cref{d:POEP}). Let $x \in \Omega$ be a critical point of $F$ such that $B_{3\sqrt{\varepsilon}}(x) \subset \Omega$. Then there exists a constant $C > 1$ such that for $\rho < \sqrt{\varepsilon}$ we have
  %
  %
  %
  % Let $\Omega,A,B,x$ be as in \cref{lem:replacement} for $R = \sqrt{\varepsilon} < 1$, furthermore assume that $x$ is a critical point of $F$ and consider the potential of the equilibrium potential $w_{A,\Omega}$, see \cref{d:POEP}.
  % Then there exists a constant $C(n,F) > 1$ such that the following holds for $\rho \leq \frac{1}{C} \sqrt{\varepsilon}$
  \begin{align*}
    w_{A,\Omega}(x) - C \bigg ( \frac{\rho}{\sqrt{\varepsilon}}\bigg )^{\alpha} (w_{A,\Omega}(x)+1)
    &\leq
    \frac{{\displaystyle \int} h_{A,\Omega^c} h_{B_\rho(x),A} d\mu_\varepsilon}{\capacity(B_\rho(x),A;\Omega)} \\
    &\leq
    w_{A,\Omega}(x) + C \bigg ( \frac{\rho}{\sqrt{\varepsilon}}\bigg )^{\alpha} (w_{A,\Omega}(x)+1).
  \end{align*}
\end{lemma}
\begin{proof}
  From \cref{lem:POEP} we get
  \begin{align*}
    \int_{\partial B_\rho(x)} w_{A,\Omega}(y) e^{-F(y)/\varepsilon} d\mu_{B_\rho(x),A}(y) = \int_{(A \cup B)^c} h_{A,\Omega^c}(z) h_{B_\rho(x),A}(z) d\mu_\varepsilon.
  \end{align*}
  We can estimate the left hand side as
  \begin{align*}
    w_{A,\Omega}(x) - \osc_{B_\rho} w_{A,\Omega}
    &\leq
    \inf_{B_\rho(x)} w_{A,\Omega}
    \leq
    \frac {\int_{\partial B_\rho(x)} e^{-F(y)/\varepsilon} w_{A,\Omega} d\mu_{B_\rho(x),A}(y)}{\int_{\partial B_\rho(x)} e^{-F(y)/\varepsilon} d\mu_{B_\rho(x),A}(y)}
    \\
    &\leq
    \sup_{B_\rho(x)} w_{A,\Omega}
    \leq
    w_{A,\Omega}(x) + \osc_{B_\rho} w_{A,\Omega}.
  \end{align*}
  We want to estimate the oscillation of $w_{A,\Omega}$ which we do by considering
  \begin{align*}
    \osc w_{A,\Omega} = \osc (w_{A,\Omega}+h_{A,\Omega^c}-h_{A,\Omega^c}) \leq \osc (w_{A,\Omega}+h_{A,\Omega^c}) + \osc(h_{A,\Omega^c}).
  \end{align*}
  Now, the oscillation of $w_{A,\Omega}+h_{A,\Omega^c}$ and $h_{A,\Omega^c}$ can estimated by \cref{lem:holder} for $\rho \leq \frac{1}{C} \sqrt{\varepsilon}$. That is,
  \begin{align*}
    \osc_{B_\rho} (w_{A,\Omega}+h_{A,\Omega^c}) + \osc_{B_\rho}(h_{A,\Omega^c})
    \leq&
    C \left ( \frac{\rho}{\sqrt{\varepsilon} }\right )^\alpha \sup_{B_{\sqrt{\varepsilon} }}(w_{A,\Omega}+h_{A,\Omega^c})
    \\
    &+
    C \left ( \frac{\rho}{\sqrt{\varepsilon} }\right )^\alpha \sup_{B_{\sqrt{\varepsilon} }}(h_{A,\Omega^c}).
  \end{align*}
  We apply \cref{l:harnack} to replace the supremums on the right hand side with the value at $x$ as both $w_{A,\Omega}+h_{A,\Omega^c}$ and $h_{A,\Omega^c}$ satisfies the Harnack inequality (see \cref{lem:harnack+}). That is,
  \begin{align*}
    \osc_{B_\rho} (w_{A,\Omega}+h_{A,\Omega^c}) + \osc_{B_\rho}(h_{A,\Omega^c})
    &\leq
    C \left ( \frac{\rho}{\sqrt{\varepsilon} }\right )^\alpha (w_{A,\Omega}(x)+h_{A,\Omega^c}(x))
    \\
    &\leq
    C \left ( \frac{\rho}{\sqrt{\varepsilon} }\right )^\alpha (w_{A,\Omega}(x)+1).
  \end{align*}
  It is easily seen that the above can be extended to $\rho \leq \sqrt{\epsilon}$ by applying \cref{lem:harnack+} again and by enlarging the constant $C$. The proof is completed by using \cref{eq:measure_energy_equality} and collecting the estimates above.
\end{proof}

\subsection{Laplace asymptotics for log-concave functions}

The assumptions \cref{eq:struc1} and \cref{eq:struc3} ensure that near critical points the potential $F$ is well approximated by convex functions.  Therefore we will need basic estimates for log-concave functions, which rather surprisingly we did not find in the literature.

\begin{lemma} \label{l:convex:lset}
  Assume  $G : \R^n \to \R $ is a convex function which has a proper minimum at the origin and $G(0)=0$. Then there exists a constant $C = C(n) > 1$ such that
  \begin{align} \label{eq:convex-1}
    \frac{1}{C} |\{G < \eps \}| \leq  \int_{\R^n} e^{-\frac{G}{\eps}} \, dx \leq C |\{G < \eps \}|.
  \end{align}
  Moreover, there is a constant $C = C(n)$ such that for all $\Lambda  > 1$, we have
  \begin{align} \label{eq:convex-2}
    \int_{\{G < \Lambda \eps \}} e^{-\frac{G}{\eps}} \, dx \geq  (1- \eta(C \Lambda^{-1})) \int_{\R^n} e^{-\frac{G}{\eps}}\,dx,
  \end{align}
  with $\eta$ as in \cref{e.eta}.
\end{lemma}
\begin{proof}
  By approximation we may assume that $G$ is smooth.  The lower bound in \cref{eq:convex-1} follows immediately from
  \begin{align*}
    \int_{\R^n} e^{-\frac{G}{\eps}} \, dx \geq \int_{\{G < \eps \} } e^{-\frac{G}{\eps}} \, dx \geq e^{-1} |\{G < \eps \}|.
  \end{align*}

  To prove the upper bound in \cref{eq:convex-1} we first show that, for all $t>0$, it holds
  \begin{align} \label{eq:convex-3}
    |\{G < 2t  \}| \leq 2^n  |\{G < t \}|.
  \end{align}
  In order to prove \cref{eq:convex-3}  it is enough to consider only the case $t = 1$ (the general case follows by considering $\tilde G = G/t$). Denote $E_1 = \{ G < 1 \}$ and $E_2 = \{ G < 2 \}$. Hence our goal is to show
  \begin{align*}
    E_2 \subset 2 E_1 = \{ 2 x : x \in E_1\} .
  \end{align*}
  Fix $\hat x \in \pa E_1$ and define $g(t) = G(t \hat x) $ for $t \geq 0$. By our assumptions, $g(t)$ is a smooth convex function satisfying $g(0)=0$ and $g(1)=1$. As such, both $g$, $g'$ are increasing functions from which we can conclude that $g'(1) \geq 1$.
  % Then  $g(0) =0 $ and $g(1) = 1$. By convexity we have $g(t) \leq t$ for $t \leq 1$, therefore $g'(1) \geq 1$. Then by convexity again, it holds $g'(t) \geq g'(1) \geq 1$ for all $t \geq 1$.
  Now, by the fundamental theorem of calculus,
  \begin{align*}
    g(2) - g(1) = \int_1^2 g'(t) \, dt \geq 1
  \end{align*}
  which gives $g(2) \geq 2$. This means that for all $\hat x \in \partial  E_1$ we have  $G(2\hat x) \geq 2$. That is, we have $E_2 \subset 2 E_1$. Thus
  \begin{align*}
    |E_2| \leq |2 E_1| \leq 2^n |E_1|
  \end{align*}
  and \cref{eq:convex-3} follows. Iterating \cref{eq:convex-3} gives
  \begin{align*}
    |\{G < 2^j \eps  \}| \leq 2^{jn}  |\{G < \eps \}|
  \end{align*}
  and hence
  \begin{align} \label{eq:convex-4}
    |\{G < \varrho \eps  \}| \leq  (2\varrho)^n |\{G < \eps \}|
  \end{align}
  for all $\varrho \geq 1$.
  We conclude the proof of the upper bound in \cref{eq:convex-1} by using \cref{eq:convex-4} as
  \begin{align*}
    \int_{\R^n} e^{-\frac{G}{\eps}} \, dx &\leq   \sum_{j=0 }^\infty  \int_{\{ j\eps \leq G < (j+1) \eps\}}  e^{-\frac{G}{\eps}} \, dx \\
    &\leq  \sum_{j=0 }^\infty  |\{ j \eps \leq G < (j+1)\eps \} |  e^{-j} \\
    &\leq 2^n \sum_{j=0 }^\infty   e^{-j} (j+1)^n   |\{G < \eps \}| \leq C(n) |\{G < \eps \}|.
  \end{align*}

  It remains to prove \cref{eq:convex-2}. Fix $\Lambda >1$. Then, for every $x \in \{ G \geq \Lambda \eps\}$, it holds
  \begin{align}\label{e.lambda.eps}
    e^{-\frac{G(x)}{\eps}} =  e^{-\frac{\Lambda G(x)}{\Lambda \eps}}  = \left(e^{-\frac{G(x)}{\Lambda \eps}}  \right)^\Lambda  =  \left(e^{-\frac{G(x)}{\Lambda \eps}}  \right)^{\Lambda-1}  e^{-\frac{G(x)}{\Lambda \eps}}  \leq e^{-\Lambda+1}  e^{-\frac{G(x)}{\Lambda \eps}}  .
  \end{align}
  Therefore we have, by \cref{e.lambda.eps,eq:convex-1,eq:convex-4},
  \begin{equation} \label{eq:convex-5}
    \begin{split}
      \int_{\{G \geq  \Lambda \eps \}} e^{-\frac{G}{\eps}} \, dx &\leq  e^{-\Lambda+1}   \int_{\{G \geq \Lambda  \eps \}} e^{-\frac{G}{ \Lambda \eps}} \, dx \leq e^{-\Lambda+1}   \int_{\R^n } e^{-\frac{G}{ \Lambda \eps}} \, dx\\
      &\leq Ce^{-\Lambda} |\{ G < \Lambda \eps \} | \\
      &\leq C e^{-\Lambda}  \Lambda^n  |\{ G < \eps \} | \\
      &\leq C e^{-\Lambda}  \Lambda^n    \int_{\R^n } e^{-\frac{G}{\eps}} \, dx
    \end{split}
  \end{equation}
  and the inequality \cref{eq:convex-2} follows by using \cref{e.eta}.
\end{proof}

\begin{lemma} \label{l:loc:convex}
  Assume  $G : \R^n \to \R $ is a function which has a proper global minimum at the origin and $G(0)=0$. Furthermore, assume there is a constant $C_0$ such that, for all $a > 0$ and $\varepsilon > 0$, it holds that
  \begin{align} \label{e.exptight}
    \int_{G > a} e^{-G/\varepsilon} dx < C_0 e^{-a/\varepsilon}.
  \end{align}
  If there is a level $\varepsilon_0 > 0$ such that $G$ is convex on the component of $\{G(x) < \varepsilon_0\}$ that contains $0$, then there is an $\varepsilon_1(n,|\{G < \varepsilon_0/2\}|) < \varepsilon_0$ and a constant $C=C(C_0,n) > 1$ such that, for all $\varepsilon < \varepsilon_1$, it holds that
  \begin{align*}
    C^{-1} |\{G < \eps \}| \leq  \int_{\R^n} e^{-\frac{G}{\eps}} \, dx \leq C |\{G < \eps \}|.
  \end{align*}
\end{lemma}
\begin{proof}
  Since $G$ is convex in the level set $\{G < \varepsilon_0\}$, we know that the level set $\{G \leq \varepsilon_0/2\}$ is convex and as such we can extend the function $G$ outside that level set to a globally convex function. This allows us to apply \cref{l:convex:lset} and obtain
  \begin{align} \label{e.head.estimate}
    \frac{1}{C} |\{G < \eps \}| \leq  \int_{\{G < \varepsilon_0/2\}} e^{-\frac{G}{\eps}} \, dx \leq C |\{G < \eps \}|.
  \end{align}
  Now, split the integral as
  \begin{align*}
    \int e^{-G/\varepsilon} dx = \int_{G \leq \varepsilon_0/2} e^{-G/\varepsilon} dx + \int_{G > \varepsilon_0/2} e^{-G/\varepsilon} dx.
  \end{align*}
  From \cref{e.head.estimate} it follows that it suffices to bound the second integral on the right hand side.
  Using \cref{e.exptight} for $a = \varepsilon_0/2$ we get
  \begin{align*}
    \int_{G > \varepsilon_0/2} e^{-G/\varepsilon} dx \leq C_0 e^{-(\varepsilon_0/2)/\varepsilon}.
  \end{align*}
  Since $G$ is convex in the level set $\{G < \varepsilon_0/2\}$, which is again convex, it follows  that we can construct a conical function $\tilde G$ as follows: For any $\hat x \in \partial \{G < \varepsilon_0/2\}$ define $\tilde G(t \hat x/\|\hat x\|) = t \varepsilon_0/2$. The level sets of $\tilde G$ satisfy, for $\varepsilon < \varepsilon_0/2$,
  \begin{align*}
    |\{\tilde G < \varepsilon\}| \leq |\{G < \varepsilon\}|.
  \end{align*}
  However,
  \begin{align*}
    |\{\tilde G < \epsilon\}| = \left ( \frac{\varepsilon}{\varepsilon_0/2}\right )^n |\{\tilde G < \varepsilon_0/2\}| = \left ( \frac{\varepsilon}{\varepsilon_0/2}\right )^n |\{G < \varepsilon_0/2\}|.
  \end{align*}
  Now, we can choose $\varepsilon_1(n,C_0,|\{G < \varepsilon_0/2\}|) < \varepsilon_0/2$ such that for $\varepsilon < \varepsilon_1$ we have
  \begin{align*}
    e^{-(\varepsilon_0/2)/\varepsilon} \leq \left ( \frac{\varepsilon}{\varepsilon_0}\right )^n |\{G < \varepsilon_0/2\}|.
  \end{align*}
  This means that for $\varepsilon < \varepsilon_1$ we also have
  \begin{align*}
    \int_{G > \varepsilon_0/2} e^{-G/\varepsilon} dx \leq C_0 |\{G < \varepsilon\}|
  \end{align*}
  which together with \cref{e.head.estimate} completes the proof.
  %
  %
  % We first choose $a$ such that
  % \begin{align*}
  % 	\int_{G > a} e^{-G/\varepsilon} dx \leq C_0|\{G < \eps \}|
  % \end{align*}
  % which by \cref{e.exptight} is satisfied, if
  % \begin{align} \label{e:exptight}
  % 	e^{-a/\varepsilon} \leq |\{G < \varepsilon\}|.
  % \end{align}
  % However by the fact that $G$ is $C^2$ and $0$ a minimum, there is a quadratic function touching $G$ from above close to $0$. Thus, if we choose $\varepsilon_1(c_0,\varepsilon_0)<\varepsilon_0$ such that
  % \begin{align*}
  % 	e^{-a/\varepsilon} \leq c_0 \varepsilon^{n/2}
  % \end{align*}
  % for all $\varepsilon < \varepsilon_1$
  % then \cref{e:exptight} holds.
  % From \cref{l:convex:lset} and the above we get the lemma.
\end{proof}

We conclude this section with the following  technical lemma which is useful when we study the potential near critical points.
\begin{lemma}
  \label{lem:tech-loc}
  Assume  $G : \R^n \to \R $ is a convex function which has a proper minimum at the origin and $G(0)=0$. Let $\omega: [0,\infty) \to [0,\infty)$ be as in \cref{eq:struc1,eq:struc3}. Then for all $\delta \leq \delta_0$, we have
  \begin{align*}
    \int_{\{ G < \delta \} } e^{-\frac{G(x)}{\eps}} e^{ \frac{\pm \omega(G(x))}{\eps}} \, dx \simeq  \int_{\R^n} e^{-\frac{G(x)}{\eps}} \, dx.
  \end{align*}
\end{lemma}

\begin{proof}
  Denote $\Lambda_\eps = \frac{\eps_1}{\eps}$ with $\eps_1$ as in \cref{lem:geometric-1}. From \cref{lem:geometric-2} we know that $\Lambda_\eps \to \infty$ as $\eps \to 0$.
  Now, by \cref{eq:convex-2}  in \cref{l:convex:lset} and \cref{lem:geometric-2}, we get
  \begin{align} \label{lem:geometric-3}
    \int_{\{G < \eps_1\}} e^{-\frac{G(x)}{\eps}} e^{\frac{\omega(G(x))}{\eps}} \, dx \simeq \int_{\{G < \Lambda_\eps \eps \}} e^{-\frac{G(x)}{\eps}}  \, dx \simeq  \int_{\R^n} e^{-\frac{G(x)}{\eps}}  \, dx.
  \end{align}
  The lower bound  follows immediately from this.
  In order to prove the upper bound, note that $\omega(s) \leq s/2$ for all $s \leq \delta_0$ by  assumption. Therefore we can repeat the argument in \cref{eq:convex-5} to get
  \begin{align*}
    \int_{\{ \eps_1 < G < \delta \} } e^{-\frac{G(x)}{\eps}} e^{\frac{\omega(G(x))}{\eps}} \, dx \leq	 \int_{\{ G > \Lambda_\eps \eps  \} } e^{-\frac{G(x)}{2\eps}}  \, dx \leq \eta(C \Lambda_\eps^{-1})  \int_{\R^n } e^{-\frac{G}{\eps}} \, dx,
  \end{align*}
  which together with \cref{lem:geometric-3} yields the upper bound.
\end{proof}

\section{Proofs of \cref{thm1} and \cref{thm2}}

In this section we prove the capacity estimate in \cref{thm1} and exit time estimate in \cref{thm2}. Before we begin, we would like to remind the reader that, as in \cref{sec:technical}, we will assume that $F$ satisfies our structural assumptions and that all constants depend on the data, see the paragraph after \cref{e.etahat}.

We first study the geometric quantities $d_\eps(A,B;\Omega)$ and $V_\eps(A,B;\Omega)$ defined in \cref{eq:struc5} and  \cref{eq:struc6} and give a more explicit, but less geometric, characterization. The characterization for the geodesic distance $d_\eps(A,B;\Omega)$  turns out to be much easier than for the separating surface $V_\eps(A,B;\Omega)$ and therefore we prove it first.
\begin{proposition}
  \label{lem:geometric-d}
  Assume that $x_a$ and $x_b$ are local minimum points of $F$, let $U_a$ and $U_b$ be the islands, i.e., the components of the set $U_{-\delta/3}$, containing $B_\eps(x_a)$ and $B_\eps(x_b)$ respectively. Assume that  $z$ is a saddle point in $Z_{x_a,x_b}$, such that the  bridge $O_{z, \delta}$ connects $U_a$ and  $U_b$, and  denote $\Omega = U_a \cup U_b \cup O_{z, \delta}$. Then it holds for $g_z$ given in \cref{eq:struc3} that
  \begin{align*}
    d_{\eps}(B_\eps(x_a),B_\eps(x_b); \Omega) \simeq e^{\frac{F(z)}{\eps}}  \int_{\R} e^{-\frac{g_z(x_1)}{\eps}} \, dx_1 .
  \end{align*}
\end{proposition}

\begin{proof}
  Begin by denoting $g = g_z$ and let us first prove the lower bound, i.e.,
  \begin{equation} \label{lem:geometric-d-1}
    d_{\eps}(B_\eps(x_a),B_\eps(x_b); \Omega) \geq (1- \hat \eta(C\eps)) e^{\frac{F(z)}{\eps}} \int_{\R} e^{-\frac{g(x_1)}{\eps}} \, dx_1 .
  \end{equation}
  To this aim we choose a smooth curve $\gamma \in \mathcal{C}(B_\eps(x_a),B_\eps(x_b); \Omega)$ which, by assumptions, intersects the bridge $O_{z, \delta}$. We may choose the coordinates in $\R^n$ such that  $z= 0$ and
  \begin{align*}
    O_{\delta} = O_{z, \delta}  = \{ x_1 :g(x_1) <  \delta\} \times \{ x' \in \R^{n-1} : G(x') <  \delta\} .
  \end{align*}
  Moreover, by changing the potential from $F$ to $F - F(z)$ we may assume that $F(0)= 0$. Note that then it holds $U_a, U_b \subset \{ F <-\delta/3\}$ as $F(z) \geq F(x_a;x_b)$.

  Let us fix $s \in \R$ such that $g(s) < \frac{\delta}{10}$ and denote $\Gamma_s = \{ s\} \times \{ G < \delta\} \subset O_\delta$.  By the assumption \cref{eq:struc3} we have
  \begin{align*}
    F > - \frac{\delta}{4} \qquad \text{on } \, \Gamma_s.
  \end{align*}
  In particular,  since $U_a, U_b \subset \{ F <-\delta/3\}$, then  the surface $\Gamma_s $ does not intersect $U_a$ or $U_b$. Thus we conclude that  every  $\gamma \in  \mathcal{C}(B_\eps(x_a),B_\eps(x_b); \Omega)$  intersects $\Gamma_s $, i.e., $\Gamma_s \in \mathcal{S}(B_\eps(x_a),B_\eps(x_b); \Omega)$. Let us denote the projection to the $x_1$-axis by $\pi_1 : \R^n \to \R$, i.e. $\pi_1(x) = x_1$. From the previous discussion we conclude that $s \in \pi_1\big(\gamma([0,1]) \cap O_\delta\big)$. This holds  for every $s \in \{g < \delta/10\}$, and therefore
  \begin{align} \label{lem:geometric-d-3}
    %\{g < \delta/10\} \subset \pi_1\big(\gamma([0,1])\big). \\
    \{g < \delta/10\} \subset \pi_1\big(\gamma([0,1]) \cap O_\delta\big).
  \end{align}
  Now the assumption \cref{eq:struc3} implies that, in the set $O_{ \delta}$, it holds that
  \begin{equation} \label{lem:geometric-d-4}
    \begin{split}
      F(x) &\geq -g(x_1) - \omega(g(x_1)) + G(x') -\omega(G(x')) \\
      &\geq -g(x_1) - \omega(g(x_1))  +\frac12 G(x') \geq  -g(x_1) - \omega(g(x_1))  .
    \end{split}
  \end{equation}
  Then for  $\gamma_1 = \pi_1(\gamma)$ we  have by   \cref{lem:geometric-d-3} and \cref{lem:tech-loc} that
  \begin{align*}
    \int_{\{t:\gamma(t) \in O_{ \delta}\}} |\gamma'| e^{\frac{F(\gamma)}{\eps}} \,  dt &\geq \int_{\{t:\gamma(t) \in O_{ \delta}\}} |\gamma_1'| e^{\frac{-g(\gamma_1)}{\eps}}e^{\frac{-\omega(g(\gamma_1))}{\eps}} \, dt \\
    &\geq \int_{\{ g < \delta/10\}}  e^{\frac{-g(x_1)}{\eps}}e^{-\frac{\omega(g(x_1))}{\eps}} \, dx_1\\
    %&\geq (1- \eta(\eps))\int_{\R} e^{-\frac{g(x_1)}{\eps}} \, dx_1.
    &\geq (1- \eta(C \eps))\int_{\R} e^{-\frac{g(x_1)}{\eps}} \, dx_1,
  \end{align*}
  proving \cref{lem:geometric-d-1}.

  To prove the upper bound, i.e.
  \begin{align*}
    d_{\eps}(B_\eps(x_a),B_\eps(x_b); \Omega) \leq (1+ \hat \eta(C \eps)) \int_{\R} e^{-\frac{g(x_1)}{\eps}} \, dx_1
  \end{align*}
  with $\hat \eta$ as in \cref{e.etahat}.
  we  denote  by $c_- < 0 < c_+$ the numbers such that $g(c_-) = g(c_+) = \delta$. We first connect the points $x_1 = (c_-, 0)$ and $x_2 = (c_+, 0)$ by a segment $\gamma_0(t) = tx_1 +(1-t)x_2$. Then it holds by the assumption  \cref{eq:struc3} and by  \cref{lem:tech-loc} that
  \begin{align*}
    \int_{\gamma_0} e^{-\frac{F(\gamma_0)}{\eps}} \, dt \leq  \int_{\{g < \delta\}}  e^{\frac{-g(x_1)}{\eps}}e^{\frac{\omega(g(x_1))}{\eps}} \, dx_1 \leq (1+\hat \eta(C\Lambda_\eps^{-1})) \int_{\R} e^{-\frac{g(x_1)}{\eps}} \, dx_1.
  \end{align*}
  We then connect $x_a$ to $x_1$ and  $x_2$ to $x_b$ with smooth curves  $\gamma_1, \gamma_2 \subset \{x \in \Omega :  F(x )< -\delta/3\}$. Since it holds $g(t) \leq C|t|$ we have $|\{g <\eps\}|\geq c \, \eps$.  Therefore it holds by \cref{l:convex:lset} that
  \begin{align*}
    \int_{\gamma_i} |\gamma_i'|e^{-\frac{F(\gamma_i)}{\eps}} \, dt
    &\leq
    e^{\frac{-\delta}{3\eps}} \int_{\gamma_i} |\gamma_i'|\, dt
    \leq
    C e^{\frac{-\delta}{3\eps}}
    \\
    &\leq
    \hat \eta(C \eps) |\{g <\eps\}| \leq \hat \eta(C \eps) \int_{\R} e^{-\frac{g(x_1)}{\eps}} \, dx_1.
  \end{align*}
  The constant in the last expression depends on the length of $\gamma_i$. We can use a similar argument as in the proof of \cref{l:capacity:rough} to bound the length of the curve. This time, we will however consider coverings with balls of size comparable to $\delta$, as we are in the level set $\{F < -\delta/3\}$ we have some room to replace our curve with another curve which has a length depending on $\delta$ and $R$, while still retaining the same upper bound as above.

  The upper bound now follows by joining the paths $\gamma_1, \gamma_0$ and $\gamma_2$, thus, constructing a competitor for the geodesic length.
\end{proof}

We need to prove similar result to \cref{lem:geometric-d}, but for the separating surface. This turns out to be trickier than the previous result for paths.

\begin{proposition}
  \label{lem:geometric-V}
  Assume that $x_a,x_b,z, U_a,U_b,O_\delta$ and $\Omega$ are as in \cref{lem:geometric-d}. Then it holds for $G_z$ from \cref{eq:struc3} that
  \begin{align*}
    V_{\eps}(B_\eps(x_a),B_\eps(x_b);\Omega) \simeq e^{-\frac{F(z)}{\eps}}  \int_{\R^{n-1}} e^{-\frac{G_z(x')}{\eps}} \, dx' .
  \end{align*}
\end{proposition}

\begin{proof}
  Denote  $G_z = G$ for short. As in the proof of \cref{lem:geometric-d} we may assume that $z = 0$, $F(0) = 0$ and that
  \begin{align*}
    O_{\delta} = O_{z, \delta}  = \{ x_1 :g(x_1) <  \delta\} \times \{ x' \in \R^{n-1} : G(x') < \delta\} .
  \end{align*}

  Let us begin by proving the upper bound. In the proof of \cref{lem:geometric-d} we already	observed that the surface $\Gamma_0 = \{ 0\} \times \{ G <  \delta\} $ is in the family of separating surfaces $\Gamma_0 \in \mathcal{S}(B_\eps(x_a),B_\eps(x_b); \Omega)$.	Therefore the assumption  \cref{eq:struc3,lem:tech-loc} together with the definition of $V_\eps$ imply
  \begin{align*}
    \begin{split}
      V_{\eps}(B_\eps(x_a),B_\eps(x_b);\Omega) &\leq \int_{\Gamma_0 } e^{-\frac{F}{\eps}} \, d\mathcal{H}^{n-1} \leq \int_{\{ G < \delta\} } e^{-\frac{G}{\eps}} e^{\frac{\omega(G)}{\eps}} \, dx' \\
      &\simeq \int_{\R^{n-1}} e^{-\frac{G(x')}{\eps}} \, dx'  .
    \end{split}
  \end{align*}
  The upper bound follows directly from this. Moreover by \cref{l:convex:lset} it holds that
  \begin{equation}
    \label{eq:geometric-V-1}
    V_{\eps}(B_\eps(x_a),B_\eps(x_b);\Omega) \leq C |\{ G < \eps\}| \leq C.
  \end{equation}

  In order to prove the lower bound we fix a small $t>0$ and  choose a smooth hypersurface $S \in \mathcal{S}(B_\eps(x_a),B_\eps(x_b);\Omega)$ such that
  \begin{align*}
    \int_{S} e^{-\frac{F}{\eps}} \, d\mathcal{H}^{n-1}   \leq V_{\eps}(B_\eps(x_a),B_\eps(x_b);\Omega) + t.
  \end{align*}
  Then $S$ divides the domain $\Omega$ into two different components, from which we denote the component containing $x_a$ by $\hat U_a$.  Note that then $\partial \hat U_a \cap \Omega \subset S$. Denote $\rho = \eps^2$. We use an idea from \cite{J} and instead of studying the set $\hat U_a$, we study the density
  \begin{align*}
    v_\rho(x) := \frac{|B_\rho(x) \cap \hat U_a|}{|B_\rho|}
  \end{align*}
  which can be written as a convolution, $v_\rho(x) = \frac{1}{|B_\rho|} ( \chi_{ \hat U_a} * \chi_{B_\rho} )$.
  To see why studying $v_\rho$ is relevant, we need some setup that we will present next.
  \begin{figure}
    \begin{center}
      \begin{tikzpicture}

        \draw (-2,-2) .. controls (0,-1) and (0,1) .. (-2,2);
        \draw (-2,1) node {$U_b$};
        \draw (2,-2) .. controls (0,-1) and (0,1) .. (2,2);
        \draw (2,1) node {$U_a$};

        \filldraw[color=gray!60, fill=gray!30] (-2,0) circle (0.5);
        \draw (-2,0) node {$x_b$};
        \draw (-2,-0.8) node {$B_\varepsilon(x_b)$};
        \filldraw[color=gray!60, fill=gray!30] (2,0) circle (0.5);
        \draw (2,0) node {$x_a$};
        \draw (2,-0.8) node {$B_\varepsilon(x_a)$};

        \draw[pattern=north west lines, pattern color=blue] (-0.8,-1.3) rectangle (0.8,1.3);
        \draw[pattern=north east lines, pattern color=blue] (-0.8,-0.4) rectangle (0.8,0.4);
        \draw (0,-1.7) node {$O_{\delta}$};

        \draw[pattern=north west lines, pattern color=blue] (-4,2) rectangle (-3.7,2.3);
        \draw (-4+0.6,2+0.1) node {$O_\delta$};

        \draw[pattern=north west lines, pattern color=blue] (-4,2-0.6) rectangle (-3.7,2.3-0.6);
        \draw[pattern=north east lines, pattern color=blue] (-4,2-0.6) rectangle (-3.7,2.3-0.6);
        \draw (-4+0.6,2+0.1-0.6) node {$\hat O$};

        \draw[pattern=north east lines, pattern color=blue] (-0.8,-0.4) rectangle (0.8,0.4);
      \end{tikzpicture}
    \end{center}
    \caption{The bridge $O_{\delta}$ connects the sets $U_a$ and $U_b$. The smaller bridge $\hat O$ has its lateral boundaries  inside $U_a \cup U_b$.}
    \label{f:localization2}
  \end{figure}
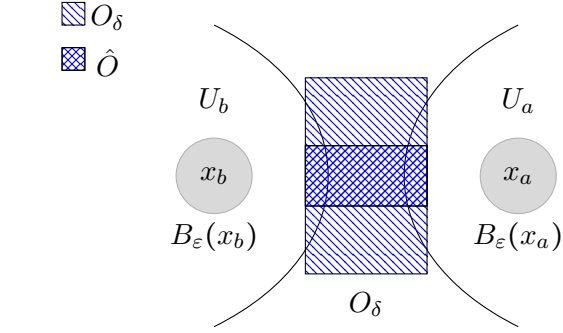
  We choose a  subset $\hat O$ of the bridge $O_{\delta}$ as
  \begin{align*}
    \hat O :=  \{ x_1 :g(x_1) <  \delta\} \times \{ x' \in \R^{n-1} : G(x') < \delta/100\},
  \end{align*}
  see \cref{f:localization2},
  and denote its lateral boundaries by $\Gamma_{-} $ and $\Gamma_{+}$, i.e.,
  \begin{align*}
    \{g=  \delta\} \times \{  G  < \delta/100\} = \Gamma_{-} \cup \Gamma_{+}.
  \end{align*}
  Now since $z \in Z_{x_a,x_b}$, we have $F(z) \leq F(x_a;x_b) + \delta/3$. Using this
  and \cref{eq:struc3} we deduce that $\Gamma_{-} \cup \Gamma_{+} \subset \{F < F(x_a;x_b) -\delta/3\}$ and therefore $\Gamma_{-} \cup \Gamma_{+}  \subset U_a \cup U_b$. Moreover, by relabeling we may assume that $\Gamma_{+} \subset U_a$ and $\Gamma_{-} \subset U_b$.
  Furthermore, by the Lipschitz-continuity of $F$ we have $|F(x) -F(y)| \leq c \eps^2$ for all $ y \in B_\rho(x)$. Note also that for all $x \in \hat O$ and $y \in B_\rho(x)$ it holds $x-y \in \Omega$.

  We will now relate $v_\rho$ to the surface integral of $S$ as follows:
  Recall that the set $\hat U_a$ has smooth boundary in $\Omega$ and thus its characteristic function is a BV-function. In particular, the derivative $|\nabla \chi_{\hat U_a}|$ is a  Radon measure in $\Omega$ and
  \begin{align} \label{e.grad.surface}
    \int_{\Omega} |\nabla \chi_{\hat U_a}| e^{-\frac{F}{\eps}}\, dx =   \int_{\partial \hat U_a \cap \Omega}  e^{-\frac{F}{\eps}}\, d\Ha^{n-1} \leq \int_{S}  e^{-\frac{F}{\eps}}\, d\Ha^{n-1} .
  \end{align}
  Using the definition of $v_\rho$ and the Lipschitzness of $F$ inside $B_\rho(x)$, we may thus estimate
  \begin{align} \label{e.nablarho.nablachi}
    \begin{split}
      \int_{\hat O} |\nabla  v_\rho| e^{-\frac{F(x)}{\eps}} \, dx &\leq \frac{1}{|B_\rho|} \int_{\hat O} e^{-\frac{F(x)}{\eps}}  \int_{\R^n}   |\nabla  \chi_{ \hat U_a}(y) | |\chi_{B _\rho}(x-y)|  \,dy dx \\
      &\leq \frac{1}{|B_\rho|} \int_{\hat O} e^{c \eps}\int_{\R^n}  e^{-\frac{F(y)}{\eps}}    |\nabla  \chi_{ \hat U_a}(y) | |\chi_{B _\rho}(x-y)|  \,dy dx \\
      &\leq (1+C\eps) \int_{\Omega}  |\nabla \chi_{\hat U_a}(y)| e^{-\frac{F(y)}{\eps}}\, dy.
    \end{split}
  \end{align}

  Putting together \cref{e.grad.surface,e.nablarho.nablachi} we see that it is enough to establish a lower bound on the integral of $|\nabla v_\rho|$ in $\hat O$. In order to achieve this, we first claim that for all $x$ such that $B_\rho(x) \subset \Omega $ we have, when $\epsilon$ is small,
  \begin{equation}
    \label{eq:geometric-V-2}
    v_\rho(x) \geq 1-  C \eps \,\, \text{ for }\, x \in  U_a \quad \text{and } \quad  v_\rho(x) \leq C \eps \,\, \text{ for }\,  x \in  U_b.
  \end{equation}
  We now complete the proof of the lower bound, using \cref{eq:geometric-V-2}, followed by the proof of \cref{eq:geometric-V-2}. Assume now \cref{eq:geometric-V-2}. Then we can use the fundamental theorem of calculus to get that for all $x' \in \{  G  < \delta/100\}$
  \begin{equation}
    \label{eq:geometric-V-3}
    1- 2C \eps \leq  \int_{\{g < \delta\}} \partial_{x_1} v_\rho(x_1,x')\, dx_1  .
  \end{equation}
  Now, arguing as in \cref{lem:geometric-d-4}  we conclude that
  \begin{align} \label{e.FGomega}
    F(x) \leq G(x') + \omega(G(x') )\qquad \text{for } \, x \in  O_{\delta}.
  \end{align}
  Multiplying and dividing with $e^{-\frac{F(x)}{\eps}}$ inside the integral in \cref{eq:geometric-V-3} and using \cref{e.FGomega} we get
  \begin{align*}
    (1- 2C \eps) e^{-(G(x')+\omega(G(x')))/\epsilon} \leq  \int_{\{G < \delta/100\}} \int_{\{g < \delta\}} |\partial_{x_1} v_\rho(x_1,x')| e^{-F(x)/\epsilon} \, dx_1.
  \end{align*}
  Integrating over $x' \in \{G < \delta/100\}$ we obtain
  \begin{align*}
    \begin{split}
      (1- 2C \eps) &\int_{\{  G  < \delta/100\}} e^{-\frac{G(x')}{\eps}}e^{-\frac{\omega(G(x'))}{\eps}}  \, dx' \\
      &\leq  \int_{\{  G  < \delta/100\}}  \int_{\{g < \delta\}} |\partial_{x_1} v_\rho(x_1,x')| e^{-\frac{F(x)}{\eps}} \, dx_1 dx'  \\
      &\leq \int_{\hat O} |\nabla  v_\rho| e^{-\frac{F(x)}{\eps}} \, dx.
    \end{split}
  \end{align*}
  The lower bound on the integral on the right hand side follows by \cref{lem:tech-loc}, i.e. we have
  \begin{equation}
    \label{eq:geometric-V-4}
    (1- \eta(\eps))  \int_{\R^{n-1}} e^{-\frac{G(x')}{\eps}}  \, dx'  \leq  \int_{\hat O} |\nabla  v_\rho| e^{-\frac{F(x)}{\eps}} \, dx .
  \end{equation}
  Now, assuming \cref{eq:geometric-V-2}, we may use \cref{eq:geometric-V-4,e.grad.surface,e.nablarho.nablachi} to get the lower bound from
  \begin{align*}
    \begin{split}
      (1- \eta(\eps))  \int_{\R^{n-1}} e^{-\frac{G(x')}{\eps}}  \, dx' &\leq (1+C\eps) \int_{S}  e^{-\frac{F}{\eps}}\, d\Ha^{n-1} \\
      &\leq (1+C\eps) \big(V_{\eps}(B_\eps(x_a),B_\eps(x_b);\Omega) + t\big)
    \end{split}
  \end{align*}
  as $t$ is arbitrarily small. Thus we obtain the lower bound, and hence in order to complete the proof it remains to prove \cref{eq:geometric-V-2}. For this, we fix $x \in  U_a \cup U_b$ such that $B_\rho(x) \subset \Omega $. By the relative isoperimetric inequality (also called Dido's problem, see for instance \cite[Theorem 3.40]{AFP} or \cite{Maz1}) and by $\rho = \eps^2$ it holds
  \begin{align*}
    \begin{split}
      \Ha^{n-1}\big(\partial \hat U_a \cap B_\rho(x)\big) &\geq c \min \big{\{} |B_\rho(x) \cap \hat U_a|^{\frac{n-1}{n} } , |B_\rho(x) \setminus  \hat U_a|^{\frac{n-1}{n} } \big{\}}\\
      &\geq c\, \eps^{2(n-1)} \min \big{\{} v_\rho(x) , 1- v_\rho(x)   \big{\}}^{\frac{n-1}{n}}.
    \end{split}
  \end{align*}
  On the other hand,  since $x \in \{ F < -\delta/3\}$ and thus $ B_\rho(x) \subset \{ F < -\delta/4 \}$,  we have by \cref{eq:geometric-V-1} that
  \begin{align*}
    \begin{split}
      \Ha^{n-1}\big(\partial \hat U_a \cap B_\rho(x)\big) &\leq  e^{-\frac{\delta}{4 \eps}}  \int_{\partial \hat U_a \cap B_\rho(x)} e^{-\frac{F}{\eps}} \, d\Ha^{n-1}  \\
      &\leq  e^{-\frac{\delta}{4 \eps}} \int_{S} e^{-\frac{F}{\eps}} \, d\mathcal{H}^{n-1}   \\
      &\leq  e^{-\frac{\delta}{4 \eps}} \big(V_{\eps}(B_\eps(x_a),B_\eps(x_b);\Omega) + t\big) \\
      &\leq C  e^{-\frac{\delta}{4 \eps}}.
    \end{split}
  \end{align*}
  By combining the two inequalities above we obtain  \cref{eq:geometric-V-2} which completes the whole proof.
\end{proof}

\subsection*{Proof of \cref{thm1}}
We consider parallel case and series case separately.
\subsubsection*{Parallel case:}
Assume that the saddle points in $F_{x_a,x_b} = \{ z_1, \dots, z_N\}$ are parallel, see \cref{f:par:ser}. Let us fix a saddle point $z_i \in F_{x_a,x_b}$ and recall the definition of the bridge $O_{z_i,\delta}$ in \cref{eq:struc4}. As before, by considering $F-F(z_i)$ instead of $F$, we may assume
\begin{align*}
  z_i = 0, \quad F(0) = 0
\end{align*}
and
\begin{align*}
  O_{\delta} := O_{0,\delta} = \{ x_1 \in \R : g(x_1) < \delta\} \times  \{ x' \in \R^{n-1} : G(x') < \delta\}.
\end{align*}
We also recall the notation $U_{-\delta/3} = \{ F < F(x_a;x_b) -\delta/3\}$.  We denote the island, i.e., the component of $U_{-\delta/3}$, which contains the point $x_a$ by $U_a$ and the island which contains the point $x_b$ by $U_b$.   Since the saddle points are parallel, the bridge $O_{\delta}$ connects the islands $U_a$ and $U_b$ which means that  the set $\Omega  = U_a \cup U_b \cup O_{\delta} $ is open and connected. By \cref{l:hab:levelset:rough} we have
$\text{osc}_{U_a}(h_{A,B}) + \text{osc}_{U_b}(h_{A,B}) \leq C \eps$. Since $h_{A,B} = 1$ in $B_\eps(x_a) \subset U_a$  and $h_{A,B} = 0$ in $B_\eps(x_b) \subset U_b$, it follows that
\begin{equation}
  \label{eq:proof-thm1-1}
  h_{A,B} \geq 1- C\eps \,\, \text{in } \, U_a \quad \text{and} \quad h_{A,B} \leq   C\eps \,\, \text{in } \, U_b.
\end{equation}

Let us choose a  subset $\hat O$ of the bridge $O_{\delta}$ as in the proof of \cref{lem:geometric-V} (see \cref{f:localization2}), i.e.
\begin{align*}
  \hat O :=  \{ x_1 :g(x_1) <  \delta\} \times \{ x' \in \R^{n-1} : G(x') < \delta/100\},
\end{align*}
and denote its lateral boundaries by $\Gamma_{-} \subset \{ x_1 <0\}$ and $\Gamma_{+}\subset \{ x_1 >0\} $, i.e.,
\begin{align*}
  \{g =  \delta\} \times \{  G  < \delta/100\} = \Gamma_{-} \cup \Gamma_{+}.
\end{align*}
Then by using \cref{eq:struc3}  and arguing as in the proof of \cref{lem:geometric-V}  we deduce that  $\Gamma_{-} , \Gamma_{+} \subset \{ F < F(x_a;x_b) -\delta/3\}$ and we may assume  $\Gamma_{+} \subset U_a$ and $\Gamma_{-} \subset U_b$.  Therefore, by \cref{eq:proof-thm1-1}, we have
that $h_{A,B} \leq C\eps $ on $\Gamma_- $ and $h_{A,B} \geq 1-  C\eps$ on $\Gamma_+$. Now, by the fundamental theorem of calculus and Cauchy-Schwarz inequality, it holds that
\begin{equation}
  \label{eq:proof-thm1-3}
  \begin{split}
    1- 2C\eps &\leq \int_{\{ g < \delta\}} \partial_{x_1} h_{A,B}(x) \,dx_1  =  \int_{\{ g <  \delta\}}\partial_{x_1} h_{A,B}(x) e^{-\frac{F(x)}{2\eps}} e^{\frac{F(x)}{2\eps}} \,dx_1  \\
    &\leq \left( \int_{\{ g <  \delta\}} |\nabla h_{A,B}(x)|^2 e^{-\frac{F(x)}{\eps}} \,dx_1 \right)^{\frac12} \left(  \int_{\{ g <  \delta\}} e^{\frac{F(x)}{\eps}} \,dx_1\right)^{\frac12}.
  \end{split}
\end{equation}
Let us next estimate the last term above. By assumption \cref{eq:struc3} we have, for $x \in \hat O$,
\begin{align*}
  F(x) \leq -g(x_1) + \omega(g(x_1) ) + G(x') + \omega(G(x')).
\end{align*}
Therefore, by \cref{lem:tech-loc,lem:geometric-d}, we can estimate
\begin{equation}
  \label{eq:proof-thm1-4}
  \begin{split}
    \int_{\{ g <  \delta\}}  e^{\frac{F(x)}{\eps}} \,dx_1 &\leq (1 + \hat \eta(C \eps))e^{\frac{G(x')}{\eps}} e^{\frac{\omega(G(x'))}{\eps}}  \int_{\{ g <  \delta\}}  e^{-\frac{g(x_1)}{\eps}} e^{\frac{ \omega(g(x_1))}{\eps}} \,dx_1 \\
    &\leq (1 + \hat \eta(C \eps))e^{\frac{G(x')}{\eps}} e^{\frac{\omega(G(x'))}{\eps}}   \int_{\R}  e^{-\frac{g(x_1)}{\eps}}  \,dx_1 \\
    &\leq (1 + \hat \eta(C \eps))e^{\frac{G(x')}{\eps}} e^{\frac{\omega(G(x'))}{\eps}}     d_\eps(B_\eps(x_a),B_\eps(x_b); \Omega).
  \end{split}
\end{equation}
We combine the inequalities \cref{eq:proof-thm1-3} and \cref{eq:proof-thm1-4} leading to (for another constant $C$)
\begin{align*}
  \int_{\{ g <  \delta\}} |\nabla h_{A,B}(x)|^2 e^{-\frac{F(x)}{\eps}} \,dx_1 \geq \frac{(1 - \hat \eta(C \eps)) }{ d_\eps(B_\eps(x_a),B_\eps(x_b); \Omega)} e^{-\frac{G(x')}{\eps}}  e^{-\frac{\omega(G(x'))}{\eps}}
\end{align*}
for all $x' \in \{ G < \delta/100\}$. By integrating over $x' \in \{ G < \delta/100\}$ we have by Fubini's theorem, \cref{lem:tech-loc,lem:geometric-V}, that
\begin{equation}
  \label{eq:proof-thm1-5}
  \begin{split}
    \int_{\hat O } |\nabla h_{A,B}|^2 e^{-\frac{F(x)}{\eps}} \,dx &\geq  \frac{(1 - \hat \eta(C \eps)) }{ d_\eps(B_\eps(x_a),B_\eps(x_b); \Omega)} \int_{\{ G < \delta/100\}} e^{-\frac{G(x')}{\eps}} e^{-\frac{\omega(G(x'))}{\eps}}     \, dx'\\
    &\geq   \frac{(1 - \hat \eta(C \eps)) }{ d_\eps(B_\eps(x_a),B_\eps(x_b); \Omega)} \int_{\R^{n-1}} e^{-\frac{G(x')}{\eps}}   \, dx' \\
    &\geq (1 - \hat \eta(C \eps)) \frac{V_\eps(B_\eps(x_a),B_\eps(x_b); \Omega) }{ d_\eps(B_\eps(x_a),B_\eps(x_b); \Omega)}.
  \end{split}
\end{equation}
Therefore, by repeating the argument for every saddle $z_i \in Z_{x_a, x_b}$ and using the fact that the bridges  $O_{z_i, \delta}$ are disjoint, we obtain after scaling back the potential
\begin{align*}
  \begin{split}
    \int_{\R^n} |\nabla h_{A,B}|^2 e^{-\frac{F(x)}{\eps}} \,dx\geq  (1 - \hat \eta(C \eps))\sum_{i=1}^N  \frac{V_\eps(B_\eps(x_a),B_\eps(x_b); \Omega) }{d_\eps(B_\eps(x_a),B_\eps(x_b); \Omega)} e^{\frac{F(z_i)}{\eps}}.
  \end{split}
\end{align*}
This yields the lower bound when the saddle points are parallel.

For the upper bound, we only give a sketch of the argument as it is fairly straightforward. The idea is to contruct a competitor $h$ in the variational characterization of the capacity, see \cref{e:cap:precise}. Let us first define $h$ in the set $U_{\delta/3} = \{ F < F(x_a;x_b) + \delta/3\}$.
%If $U_{\delta/3}$ is disconnected \textcolor{red}{What does this mean?} we consider only the component which contains $x_a$.
Since the saddle points $ Z_{x_a, x_b} = \{ z_1, \dots, z_N\}$ are parallel, it follows that the points $x_a$ and $x_b$ lie in different components of the set
\begin{align*}
  \tilde U = U_{\delta/3}  \setminus \bigcup_{i=1}^N O_{z_i,\delta},
\end{align*}
where $O_{z_i,\delta}$ is defined in \cref{eq:struc5}. Denote the components of $\tilde U$ containing $x_a$ and $x_b$ by $\tilde U_a$ and $\tilde U_b$, respectively. We define first
\begin{align*}
  h = 1 \, \text{ in } \, \tilde U_a \quad \text{and} \quad h = 0 \, \text{ in } \, \tilde U_b.
\end{align*}

Let us next fix a saddle point $z_i \in Z_{x_a, x_b}$. As before, we may again assume that
\begin{align*}
  z_i = 0, \quad F(0) = 0
\end{align*}
and
\begin{align*}
  O_{\delta} := O_{0,\delta} = \{ x_1 \in \R : g(x_1) < \delta\} \times  \{ x' \in \R^{n-1} : G(x') < \delta\}.
\end{align*}
Moreover, we may assume that
\begin{align*}
  \tilde U_a \cap \partial O_{\delta}  \subset \{ x_1 > 0\} \quad \text{ and} \quad
  \tilde U_b \cap \partial O_{\delta}  \subset \{ x_1 < 0\}.
\end{align*}
Let $c_- <0 < c_+ $  be numbers such that $g(c_-) = g(c_+) = \delta/100$. We define $h(x) = \varphi(x_1)$ in $O_{\delta}$ such that the function  $\varphi :[c_-,c_+]\to \R$ is a 			solution of the ordinary differential equation
\begin{align*}
  \frac{d}{ds} \left(\varphi'(s) e^{\frac{g(s)}{\eps}} \right) = 0 \quad \text{in } \, (c_-,c_+)
\end{align*}
with boundary values $\varphi(c_-) = 0$ and $\varphi(c_+) = 1$. We extend $\varphi$ into $\R$ by setting $\varphi(s) = 0$ for $s \leq c_-$ and $\varphi(s) = 1$ for $s \geq c_+$. It follows that for the function $h$ we have, by construction,  \cref{lem:tech-loc}, and an argument similar to the one leading to \cref{eq:proof-thm1-3}, that
\begin{align*}
  \begin{split}
    \int_{O_{\delta} } &|\nabla h|^2 e^{-\frac{F(x)}{\eps}} \, dx  \\
    &\leq \int_{\{ g<  \delta\} } | \varphi'(x_1)|^2 e^{\frac{g(x_1)}{\eps}}e^{\frac{\omega(g(x_1))}{\eps}}  \, dx_1   \int_{\{ G< \delta\} }
    e^{-\frac{G(x')}{\eps}} e^{\frac{\omega(G(x'))}{\eps}}\, dx' \\
    &\leq (1+ \hat \eta(C \eps)) \left( \int_{\R}  e^{-\frac{g(x_1)}{\eps}} \, dx_1\right)^{-1} \left(  \int_{\R^{n-1}}  e^{-\frac{G(x')}{\eps}} \, dx'\right).
  \end{split}
\end{align*}
By repeating the construction for every saddle point $z_i \in Z_{x_a, x_b}$, we obtain a function which is defined in $U_{\delta/3}$. We denote this function by $h  : U_{\delta/3} \to \R$. Note that now for $h$  the estimate  \cref{eq:proof-thm1-5} is optimal. Moreover, $h$ is Lipschitz continuous. We extend $h$ to $\R^n$ without increasing the Lipschitz constant $L$, e.g., by defining
\begin{align*}
  h(x) = \sup_{y \in U_{\delta/3}} \big( h(y) -L|x-y|\big) \quad \text{for } \, x \in \R^n \setminus U_{\delta/3}.
\end{align*}
This finally leads to the upper bound completing the proof of the parallel case, while we leave the final details on the upper bound for the reader.
\subsubsection*{Series case:}
Assume that the saddle points $Z_{x_a, x_b} = \{ z_1, \dots, z_N\}$  are in series, see \cref{f:par:ser}. We use the ordering as in  \cref{eq:struc8} and denote the points $x_i$ as in \cref{eq:struc9}. We also fix the islands, $U_{x_{i-1}}$ and $U_{x_i}$ (components of $\{ F < F(x_a; x_b) -\delta/3\}$), which are connected by the bridge $O_{z_i,  \delta}$. Again we may assume that
$z_i = 0$,  $F(0) = 0$ and that
\begin{align*}
  O_{ \delta}  =O_{z_i,  \delta} = \{x_1:  g(x_1) <  \delta \} \times \{x':  G(x_1) <  \delta \}.
\end{align*}
By \cref{l:hab:levelset:rough} we have
$\text{osc}_{U_{x_{i-1}}}(h_{A,B}) + \text{osc}_{U_{x_{i}}}(h_{A,B}) \leq C \eps$. Therefore there are numbers $c_{i-1},c_i$ such that
\begin{align*}
  h_{A,B} \geq c_{i-1} -  C\eps \,\, \text{in } \,U_{x_{i-1}} \quad \text{and} \quad h_{A,B} \leq c_i  +  C\eps \,\, \text{in } \, U_{x_{i}}.
\end{align*}
Then, using the fundamental theorem of calculus as in \cref{eq:proof-thm1-3}, we obtain
\begin{align*}
  \begin{split}
    |c_{i-1}-c_i|- 2C\eps \leq \left( \int_{\{g<\delta\} } |\nabla h_{A,B}(x)|^2 e^{-\frac{F(x)}{\eps}} \,dx_1 \right)^{\frac12} \left(  \int_{\{g<\delta\} }  e^{\frac{F(x)}{\eps}} \,dx_1\right)^{\frac12}
  \end{split}
\end{align*}
for
\begin{align*}
  (x_1,x') \in \{ g <\delta\} \times \{ G <\delta/100\}.
\end{align*}
Moreover, arguing as in \cref{eq:proof-thm1-4}, we have
\begin{align*}
  \int_{\{g<\delta\} }  e^{\frac{F(x)}{\eps}} \,dx_1 \leq (1 + \hat \eta(C \eps))e^{\frac{G(x')}{\eps}}e^{\frac{\omega(G(x'))}{\eps}}    d_\eps(x_{i-1},x_i).
\end{align*}
These together imply
\begin{align*}
  \int_{\{g<\delta\}} |\nabla h_{A,B}(x)|^2 e^{-\frac{F(x)}{\eps}} \,dx_1 \geq \frac{(1 - \hat \eta(C \eps)) (c_{i-1}-c_i)^2}{   d_\eps(x_{i-1},x_i)} e^{-\frac{G(x')}{\eps}}e^{-\frac{\omega(G(x'))}{\eps}}    .
\end{align*}
By integrating over $x' \in \{ G < \delta/100\}$ we have, by Fubini's theorem,  \cref{lem:tech-loc},  and \cref{lem:geometric-V}, that
\begin{align*}
  \int_{O_{\delta}} |\nabla h_{A,B}|^2 e^{-\frac{F(x)}{\eps}} \,dx \geq  (1 - \hat \eta(C\eps)) (c_{i-1}-c_i)^2  \frac{V_\eps(x_{i-1},x_i) }{ d_\eps(x_{i-1},x_i)}.
\end{align*}
By repeating the argument for every saddle $z_i \in Z_{x_a, x_b}$ and  using the fact that the sets $O_{z_i, \delta}$ are disjoint we obtain
\begin{align*}
  \int_{\R^n} |\nabla h_{A,B}|^2 e^{-\frac{F(x)}{\eps}} \,dx\geq  (1 - \hat \eta(C \eps)) \sum_{i=1}^N (c_{i-1}-c_i)^2  \frac{V_\eps(x_{i-1},x_i) }{ d_\eps(x_{i-1},x_i)}.
\end{align*}
Recall that the numbers $c_i$ are the approximate values of $h_{A,B}$ in the components $U_{x_{i}}$. Therefore we may choose them such that  $1 = c_0$ and $c_{N} =0$.
By denoting $y_i = c_{i-1}-c_i$ and $a_i = \frac{V_\eps(x_{i-1},x_i) }{ d_\eps(x_{i-1},x_i)}$  we may write
\begin{align*}
  \sum_{i=1}^N (c_{i-1}-c_i)^2  \frac{V_\eps(x_{i-1},x_i) }{ d_\eps(x_{i-1},x_i)} =  \sum_{i=1}^N a_i y_i^2,
\end{align*}
where we have a constraint $\sum_{i=1}^N y_i=1$. By a standard optimization argument (using Lagrange multipliers) we get that under such a constraint it holds that
\begin{align*}
  \sum_{i=1}^N a_i y_i^2   \geq \left( \sum_{i=1}^N \frac{1}{a_i}   \right)^{-1}.
\end{align*}
This yields the lower bound in the case when the  saddle points are in series. The upper bound on the other hand follows from a similar argument than in the parallel case, and we leave the details for the reader. This completes the proof in the series case, and hence the whole proof.

\hfill\qedsymbol

\subsection*{Proof of \cref{thm2}}

Let us first recall the notation related to \cref{thm2}. We assume that the local minimas $x_i$ of $F$ are  ordered such that $F(x_i) \leq F(x_j)$ if $i \leq j$, and they are grouped into sets $G_i$ such that $x_i, x_j \in G_k$ if $F(x_i) = F(x_j)$ and $x \in G_i$,  $y \in G_j$ with $i <j$ if $F(x) < F(y)$. We also denoted $F(G_i) := F(x)$ with $x \in G_i$, $S_k = \bigcup_{i=1}^k G_i$,   $G_k^\varepsilon = \bigcup_{x \in G_k} B_\varepsilon(x)$, and $S_k^\varepsilon = \bigcup_{i=1}^k G_i^\eps$.

The proof of \cref{thm2} follows from the following lemma together with \cref{lem:POEP_ratio} and \cref{thm1}.
\begin{lemma} \label{lem:l1}
  Under the assumptions of \cref{thm2}, there exists constants $C=C(F) > 1$ and $\varepsilon=\varepsilon_0(F) \in (0,1)$ such that, for all $0 < \varepsilon \leq \varepsilon_0$, we have
  \begin{align*}
    \frac{1}{C}  \sum_{x \in G_{k+1}} |O_{x,\varepsilon}| \leq e^{\frac{F(G_{k+1})}{\varepsilon}} \int h_{G_{k+1}^\varepsilon,S_k^\varepsilon} d\mu_\epsilon \leq C  \sum_{x \in G_{k+1}} |O_{x,\varepsilon}|.
  \end{align*}
\end{lemma}
We prove \cref{thm2} first, while the proof of \cref{lem:l1} is given later on.
\begin{proof}[Proof of \cref{thm2}]
  Using \cref{lem:POEP_ratio} and choosing $\rho = \varepsilon$ we obtain that, for $\varepsilon$ small enough and $x \in G_{k+1}^\varepsilon $, that
  \begin{align*}
    \E^x[\tau_{S_k^\varepsilon } \I_{\tau_{S_k^\varepsilon } < \tau_{\Omega^c}}]
    \leq C
    \frac{{\displaystyle \int} h_{{S_k^\varepsilon},{\Omega^c}} h_{G_{k+1}^\varepsilon,{S_k^\varepsilon}} d\mu_\varepsilon}{\capacity{G_{k+1}^\varepsilon}{{S_k^\varepsilon}}}
    + C \varepsilon^{\alpha/2}.
  \end{align*}
  The ratio above can be estimated by using \cref{lem:l1} and the monotonicity of the capacity. That is, the numerator can be bounded by \cref{lem:l1}, while for the capacity we have
  \begin{align} \label{eq:capacity:reduction}
    \capacity{G_{k+1}^\varepsilon }{{S_k^\varepsilon }} \geq \capacity{G_{k+1}^\varepsilon }{{G_k^\varepsilon}} \geq \max_{x \in G_k, y \in G_{k+1} } \capacity{B_\varepsilon(x)}{{B_\varepsilon(y)}}.
  \end{align}
  The claim for the parallel and series cases now follows by assuming that the maximum is attained for a pair of minimas $x_a \in G_k$, $x_b \in G_{k+1}$ and applying \cref{thm1}.
\end{proof}
%\hfill \qedsymbol

\begin{remark}
  We note that in the general case, the last inequality in \cref{eq:capacity:reduction} has the optimal dependence with respect to $\eps$ but the inequalities may differ by a constant. Essentially the inequality is sharp only in the case where only one saddle contributes to the total value of the capacity. Hence we have the sharp estimate when saddle points are parallel or in series, but in general the situation might be more complicated than that. We have illustrated this in \cref{f:capacity}, where each gray dot is a saddle at the same height, and $A,B$ produces $G_k$. Then the precise value of $\capacity{G_{k+1}^\varepsilon }{A \cup B}$ is already non-trivial to calculate.
  %	\begin{align*}
  %		\capacity{G_{k+1}^\varepsilon }{A \cup B} \geq \frac{1}{c} (\capacity{G_{k+1}^\varepsilon }{A} + \capacity{G_{k+1}^\varepsilon }{B})
  %	\end{align*}
  %	will essentially be $2$. This is because the capacitary potential $h_{G_{k+1}^\varepsilon ,A}$ will following the two red curves have to go from $0$ to $1$, as such from \cref{thm1}, the capacitary potential needs to store most of its energy in $O_1$ and in at least one of the boxes $O_2,O_3$. Essentially the same is true for $h_{G_{k+1}^\varepsilon ,A}$ and $h_{G_{k+1}^\varepsilon ,A \cup B}$ and thus all of them will be of the same size (when $\varepsilon$ small).
\end{remark}

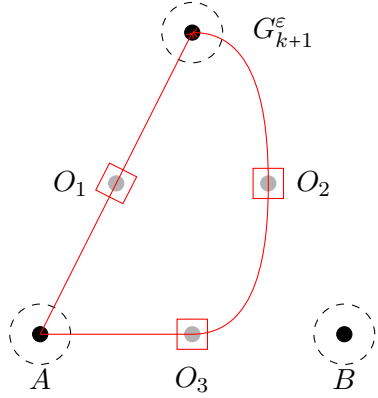
\begin{figure}
  \begin{center}
    \begin{tikzpicture}[scale=2]
      % Lets draw the three equilaterally placed minima
      \filldraw (0,2) circle (0.05);
      \draw[dashed] (0,2) circle (0.2);
      \draw (0.6,2) node {$G_{k+1}^\varepsilon $};
      \filldraw (-1,0) circle (0.05);
      \draw[dashed] (-1,0) circle (0.2);
      \draw (-1,-0.3) node {$A$};
      \filldraw (1,0) circle (0.05);
      \draw[dashed] (1,0) circle (0.2);
      \draw (1,-0.3) node {$B$};

      % Lets draw the three saddles
      \filldraw[color=gray!60] (-0.5,1) circle (0.05);
      \draw (-0.5-0.3,1) node {$O_1$};
      \filldraw[color=gray!60] (0.5,1) circle (0.05);
      \draw (0.5+0.3,1) node {$O_2$};
      \filldraw[color=gray!60] (0,0) circle (0.05);
      \draw (0,-0.3) node {$O_3$};

      % curve going right
      \draw[->,color=red] (-1,0) to[out=0,in=-180] (0,0) to[out=0,in=-90] (0.5,1) to[out=90,in=0] (0,2);
      % curve going left
      \draw[->,color=red] (-1,0) to (0,2);

      % Draw the three capacity boxes Q_1 to Q_3

      \draw[red,rotate around={63:(-0.5,1)}] (-0.5-0.1,1+0.1) rectangle (-0.5+0.1,1-0.1);
      \draw[red,rotate around={0:(0.5,1)}] (0.5-0.1,1+0.1) rectangle (0.5+0.1,1-0.1);
      \draw[red,rotate around={0:(0,0)}] (-0.1,0.1) rectangle (0.1,-0.1);

    \end{tikzpicture}
  \end{center}
  \caption{Geometric view of multiple minima at same height and multiple saddles at the same height}
  \label{f:capacity}
\end{figure}

\begin{proof}[Proof of \cref{lem:l1}]

  First we will prove a localization estimate for exponential integrals. Consider a set $0 \in O$ and a function $f$ such that $f(0) = l$ is a proper local minimum and that $f$ is locally convex around $0$. Then there exists an $\varepsilon_0$ such that, for any $\varepsilon < \varepsilon_0$,
  \begin{align} \label{e:loc:expint}
    \int_O e^{-f(x)/\varepsilon} dx \leq C e^{-l/\varepsilon}|\{f < \varepsilon\} \cap O|.
  \end{align}
  We will first prove \cref{e:loc:expint} and then repeatedly apply it to prove \cref{lem:l1}.

  In order to prove \cref{e:loc:expint}, we begin by rescaling such that $l = 0$. Then we extend $f$ outside $O$ as $+\infty$ and call this extended function $\hat f$. We first prove
  \begin{align*}
    \int_{\{\hat f > a\}} e^{-\hat f(x)/\varepsilon} dx \leq c e^{-a/\varepsilon}
  \end{align*}
  which, by the definition of $\hat f$, is equivalent to
  \begin{align*}
    \int_{\{f > a\} \cap O} e^{-f(x)/\varepsilon} dx \leq c e^{-a/\varepsilon}.
  \end{align*}
  This now follows from \cref{l:loc:convex} by using $|O| < \infty$ and observing that $\hat f$ satisfies the assumptions of \cref{l:loc:convex}. Hence we observe \cref{e:loc:expint}.

  Consider now the set
  \begin{align*}
    U_{-\delta_2/3} \equiv \{y: F(y) \leq F(G_{k+1} ;S_k )-\delta_2/3\}
  \end{align*}
  and let $U_i$ be the component of $U_{-\delta_2/3}$ containing $x_i$. We split
  \begin{align*}
    \int h_{G_{k+1}^\varepsilon,S_k^\varepsilon} e^{-F/\varepsilon} dx = \int_{U_{-\delta_2/3}^c}h_{G_{k+1}^\varepsilon,S_k^\varepsilon} e^{-F/\varepsilon} dx + \int_{U_{-\delta_2/3} \setminus S_k^\varepsilon }h_{G_{k+1}^\varepsilon,S_k^\varepsilon} e^{-F/\varepsilon} dx,
  \end{align*}
  where complement is understood with respect to the domain $\Omega$.
  By assumptions \eqref{eq:struc1} and \eqref{eq:struc3} on $F$ it holds that
  \begin{equation} \label{eq:lead00}
    F(G_{k+1} ;S_k ) \geq F(G_{k+1}) + \tfrac{2}{3} \delta_2.
  \end{equation}
  Also by the quadratic growth \eqref{eq:struc0} we can bound the first integral as
  \begin{align*}
    \int_{U_{-\delta_2/3}^c}h_{G_{k+1}^\varepsilon,S_k^\varepsilon} e^{-F/\varepsilon} dx \leq C e^{-(F(G_{k+1} ;S_k )-\delta_2/3)/\varepsilon} \leq C e^{-\frac{\delta_2}{3\eps}} e^{-F(G_{k+1})/\varepsilon},
  \end{align*}
  which shows that the first integral is neglible in the final estimate. For the second integral we further split
  \begin{align*}
    \int_{U_{-\delta_2/3} \setminus S_k^\varepsilon }h_{G_{k+1}^\varepsilon,S_k^\varepsilon} e^{-F/\varepsilon} dx = \sum_{i} \int_{U_i \setminus S_k^\varepsilon }h_{G_{k+1}^\varepsilon,S_k^\varepsilon} e^{-F/\varepsilon} dx.
  \end{align*}

  We will now consider all the different components $U_i$ depending on what minimas they contain. We start with the components $U_i$ that do not intersect $S_k^\varepsilon \cup G_{k+1}^\varepsilon$. Then all local minimas in $U_i$ are larger than  $F(G_{k+1})$, and hence from \cref{e:loc:expint} we get that there exists a constant $C$ such that
  \begin{align*}
    \int_{U_i} h_{G_{k+1}^\varepsilon,S_k^\varepsilon} e^{-F/\varepsilon} dx \leq C e^{-F(G_{k+2})/\varepsilon} \leq Ce^{-\delta_2/\varepsilon} e^{-F(G_{k+1})/\varepsilon},
  \end{align*}
  where the last inequality follows from \eqref{eq:struc10}. This shows that also this term is neglible.  Consider next the component $U_i$ that intersects $G^\varepsilon_{k+1}$ but do not intersect $S_k^\varepsilon $. In this case, by \cref{e:loc:expint} and \cref{l:convex:lset,l:hab:levelset:rough}, we have
  \begin{align} \label{eq:lead1}
    \frac{1}{C} \sum_{x \in G_{k+1}  \cap U_i} |O_{x,\varepsilon}| \leq e^{\frac{F(G_{k+1})}{\varepsilon}} \int_{U_i} h_{G_{k+1}^\varepsilon,S_k^\varepsilon} e^{-F/\varepsilon} dx
    \leq C \sum_{x \in G_{k+1}  \cap U_i} |O_{x,\varepsilon}|
  \end{align}
  providing us the leading term that contributes to the final estimate.

  Consider next a component $U_i$ such that $U_i \cap S_k^\varepsilon  \neq \emptyset$. Since $U_i$ is a component of $U_{-\delta_2/3}$, it follows from $U_i \cap S_k \neq \emptyset$ that  $F(y;S_k) \leq F(y;G_{k+1})-\delta_2/3 \leq F(y;G_{k+1})$ in $U_i$.   Therefore we have, by \cref{l:hab:bound}, in $U_i$ that
  \begin{align*}
    h_{G_{k+1}^\varepsilon,S_k^\varepsilon} \leq C \varepsilon^{q} e^{-(F(y;G_{k+1} )-F(y;S_k ))/\varepsilon}.
  \end{align*}
  Hence, for $q\in \R$, we obtain
  \[
  \int_{U_i } h_{G_{k+1}^\varepsilon,S_k^\varepsilon} e^{-F/\varepsilon} dx \leq   \varepsilon^{q}  \int_{U_i }e^{-(F(y;G_{k+1})-F(y;S_k))/\varepsilon} e^{-F(y)/\varepsilon} dy.
  \]
  In order to compute the integral on the right hand side we study the infimum value of the function $f(y) = F(y;G_{k+1} )-F(y;S_k )+F(y)$. Clearly, the infimum is attained at an interior point of $U_i$, denoted by $x_i$. It follows that then $x_i$ is necessarily a local minimum point of $F$. By above considerations, we also have $F(y;S_k ) < F(y;G_{k+1})$ for all $y \in U_i$, and thus we may deduce that $x_i \notin G_{k+1}$. If now $x_i \in S_k$,  then $F(x_i) = F(x_i;S_k)$ and thus, by the definition of $f$ and by \eqref{eq:lead00},
  \begin{align*}
    \inf_{U_i} f(y) = f(x_i) \geq F(x_i;G_{k+1}) \geq F(S_k ;G_{k+1}) \geq  F(G_{k+1}) + \tfrac{2}{3} \delta_2.
  \end{align*}
  It remains to study the case where $x_i \in G_j$ for some $j \geq k+2$. In this case we apply
  \begin{align*}
    \inf_{U_i} f(y) =  f(x_i) = \overbrace{F(x_i;G_{k+1})-F(x_i;S_k)}^{\geq 0}+F(G_{j})  \geq  F(G_{k+2}) \geq F(G_{k+1}) + \delta_2,
  \end{align*}
  where the last inequality follows from \eqref{eq:struc10}. Therefore we can conclude that,  for $\delta_3 = \tfrac{2}{3} \delta_2$, it holds that
  \[
  \int_{U_i } \varepsilon^{q} e^{-(F(y;G_{k+1})-F(y;S_k))/\varepsilon} e^{-F(y)/\varepsilon} dy
  \leq
  C\varepsilon^{q} e^{-\frac{\delta_3}{\eps}} e^{-F(G_{k+1})/\varepsilon}.
  \]
  Consequently, the component $U_i$ satisfying $U_i \cap S_k^\varepsilon  \neq \emptyset$ does not contribute either. The proof is hence completed by \cref{eq:lead1} and by the fact that the integral over the remaining components are neglible whenever  $\varepsilon$ is small enough.
\end{proof}

\section*{Acknowledgments}
B.A. was supported by the Swedish Research Council dnr: 2019-04098.
V.J. was supported by the Academy of Finland grant 314227. We would like to thank Aapo Kauranen and Toni Ikonen for helpful discussion on geometric function theory and for letting us know   the references   \cite{G,Z}.

\end{document}